\documentclass[reqno,11pt]{amsart}
\usepackage{a4wide,eucal,enumerate,mathrsfs}
\usepackage[dvipsnames]{xcolor} 
\usepackage[normalem]{ulem}
\usepackage{amsmath,amssymb,amsthm} 
\usepackage[latin1]{inputenc}
\usepackage{epsfig}
\numberwithin{equation}{section}

\newtheorem{theorem}{Theorem}[section]

\newtheorem{lemma}[theorem]{Lemma}
\newtheorem{proposition}[theorem]{Proposition}

\theoremstyle{remark}

\DeclareMathOperator{\diam}{diam}
\DeclareMathOperator{\dist}{dist}

\DeclareMathOperator{\interior}{int}

\newcommand{\N}{\mathbb{N}}

\newcommand{\R}{\mathbb{R}}

\newcommand{\mm}{\mathbf{m}}

\renewcommand{\d}{{\mathrm d}}

\def\dist{{\mathop\mathrm{\,dist\,}}}

\def\bint{{\ifinner\rlap{\bf\kern.35em--}
\int\else\rlap{\bf\kern.45em--}\int\fi}\ignorespaces}

\def\bbint{{\ifinner\rlap{\bf\kern.35em--}
\hspace{0.078cm}\int\else\rlap{\bf\kern.45em--}\int\fi}\ignorespaces}

\def\diam{{\mathop\mathrm{\,diam\,}}}

\def\dfrac{\displaystyle\frac}

\def\bint{{\ifinner\rlap{\bf\kern.35em--}
\int\else\rlap{\bf\kern.45em--}\int\fi}\ignorespaces}

\begin{document}

\title[A necessary condition for Sobolev extension domains in higher dimensions]
{A necessary condition for Sobolev extension domains in higher dimensions}

\author{Miguel Garc\'ia-Bravo}

\address{Departamento de An\'alisis Matem\'atico y Matem\'atica Aplicada, Facultad de Ciencias Matem\'aticas, Universidad Complutense, 28040, Madrid, Spain}

\email{miguel05@ucm.es}
\author{Tapio Rajala}
\author{Jyrki Takanen}

\address{University of Jyvaskyla \\
         Department of Mathematics and Statistics \\
         P.O. Box 35 (MaD) \\
         FI-40014 University of Jyvaskyla \\
         Finland}

\email{tapio.m.rajala@jyu.fi}
\email{jyrki.j.takanen@jyu.fi}

\date{\today}


\begin{abstract}
We give a necessary condition for a domain to have a bounded extension operator from $L^{1,p}(\Omega)$ to $L^{1,p}(\mathbb R^n)$ for the range $1 < p < 2$. The condition is given in terms of a power of the distance to the boundary of $\Omega$ integrated along the measure theoretic boundary of a set of locally finite perimeter and its extension. This generalizes a characterizing curve condition for planar simply connected domains, and a condition for $W^{1,1}$-extensions. We use the necessary condition to give a quantitative version of the curve condition. We also construct an example of an extension domain that is homeomorphic to a ball and has $n$-dimensional boundary.
\end{abstract}


\maketitle

\tableofcontents

 \section{Introduction}

A domain $\Omega \subset \mathbb R^n$ is called a $W^{k,p}$-extension domain, if we can extend each Sobolev function $u \in W^{k,p}(\Omega)$ to a global Sobolev function $u \in W^{k,p}(\mathbb R^n)$ so that the Sobolev norm of the extension is at most a constant times the norm of the original function. Sobolev extension domains are interesting in several fields of analysis because on those one can use many functional-analytic tools that are classically available for functions defined on the whole space. Examples of Sobolev-extension domains include Lipschitz domains  \cite{C1961,stein} and more generally, $(\varepsilon,\delta)$-domains \cite{jo1981}.
For our context, the Lipschitz and $(\varepsilon,\delta)$ results should be seen as sufficient conditions on the boundary of the domain for the extendability of Sobolev functions. In this paper, we continue investigating the converse direction by finding a new necessary condition for extendability.

Several necessary geometric conditions on the boundary of Sobolev extension domains are already known. For instance, all  Sobolev extension domains have positive densities at all the points belonging to them (this is usually referred as to satisfy a measure density condition, \cite{K1990,HKT2008}). Then, by the Lebesgue differentiation theorem we must have that their boundaries are of zero Lebesgue measure. In general we cannot improve this to a non-trivial dimension upper bound on the boundary of a Sobolev extension domain:  take for example $\Omega = [0,1]^n \setminus C^n$ with $C$ a Cantor set with zero Lebesgue measure and so that $\dim_\mathcal{H}(C) = 1$. 

However, one can still meaningfully study the dimension of the boundary of extension domains. One approach is to limit the topology or other properties of the domain, and another one is to investigate only those points that are more relevant for the extendability.
The second approach leads to the study of the size of the set of two-sided points of the boundaries of Sobolev extension domains (that is, points where the boundary might self-intersect and hence can be approached from two different sides in the domain). In the case $p\geq n$ we have that $W^{1,p}$-extension domains are quasiconvex (see \cite[Theorem 3.1]{K1990}) and then the set of two-sided points must be empty. The case  $1 \le p  < n$ is more interesting and has been investigated in \cite{Jyrki,GBRT2021}, where bounds on the Hausdorff dimension of the set of two-sided points are found.

Non-trivial dimension upper bounds for the whole boundary have been obtained only in the special case of planar bounded simply connected extension domains \cite{LRT2020}. These bounds are based on the porosity of the boundary that is implied by the geometric characterizations of bounded simply connected planar Sobolev extension domains, see \eqref{eq:Shvartsman} and \eqref{eq:characterization_1<p<2} below.
The first such characterizations established that a bounded simply connected domain $\Omega\subset\R^2$ is a $W^{1,2}$-extension domain if and only if $\Omega$ is a quasidisk (see \cite{golavo1979,gore1990,govo1981,jo1981}). 

In the case $2<p<\infty$, Shvartsman \cite{Shvartsman} proved that a bounded finitely connected domain $\Omega\subset\R^2$ is a $W^{1,p}$-extension domain  if and only if for some $C> 1$ the following condition is satisfied: for every $x, y \in \Omega$ there exists a rectifiable curve $\gamma\subset\Omega$ joining $x$ and $y$ so that 
\begin{equation}\label{eq:Shvartsman}
\int_{\gamma} \dist (z,\partial\Omega)^{\frac{1}{1-p} }\,ds(z)\leq C|x-y|^{\frac{p-2}{p-1}}.
\end{equation}
Let us mention that in \cite{ShvartsmanZobin}, the curve condition \eqref{eq:Shvartsman} was also shown to  characterize $L^{k,p}$-extension domains for every $2<p<\infty$ and $k\in\N$. Here we define the homogeneous Sobolev space $L^{k,p}(\Omega)$ to be the space of locally integrable functions whose distributional partial derivatives belong to $L^p(\Omega)$.

Finally, for the case $1<p<2$ the following result is proved in \cite{KRZ2015}: a bounded simply connected domain $\Omega\subset\R^2$ is a $W^{1,p}$-extension domain  if and only if there exists $C> 1$ such that for
every $x, y \in \R^2\setminus \Omega$  there exists a  curve $\gamma\subset\R^2\setminus \Omega$ connecting $x$ and $y$ such that  
\begin{equation}\label{eq:characterization_1<p<2}
    \int_{\gamma} \dist (z,\partial\Omega)^{1-p }\,ds(z)\leq C|x-y|^{2-p}.
\end{equation}
In Theorem \ref{thm:necessity} we generalize the condition \eqref{eq:characterization_1<p<2} to higher dimensions; still for the range $1 < p < 2$ of exponents. Before stating our result, let us look at the limiting case $p=1$ that partly motivates our formulation.

In the case of a bounded simply connected planar domain  $\Omega$, by the results from \cite{KRZ}, we know that  $\Omega$ is a $W^{1,1}$-extension domain if and only if for every $x,y\in \Omega^c$ there exists a curve  $ \gamma \subset \Omega^c$ connecting $x$ and $y$ with 
\begin{equation}\label{eq:curvep1}
\ell(\gamma) \le C|x-y|, \text{ and } \mathcal H^1(\gamma \cap \partial \Omega) = 0.
\end{equation}
In other words, the correct limit of the term $\dist (z,\partial\Omega)^{1-p}$ in \eqref{eq:characterization_1<p<2} is $1/\chi_{\mathbb R^2\setminus \Omega}(z)$ as $p \searrow 1$.
The characterizing property \eqref{eq:curvep1} can also be seen as a combination of earlier results on $BV$-extension domains and the following more general planar result \cite{GBR2021}: A bounded $BV$-extension domain $\Omega \subset \mathbb R^2$ is a $W^{1,1}$-extension domain if and only if
the $1$-dimensional measure of the set
\[
 \partial \Omega \setminus \bigcup_{i \in I} \overline{\Omega_i}
\]
intersected with any Lipschitz curve is zero, where $\{\Omega_i\}_{i\in I}$ are the connected components of $\mathbb R^2 \setminus \overline{\Omega}$.
Recall that the space $BV(\Omega)$ consists of integrable functions $u\in L^{1}(\Omega)$ whose total variation 
\[
\Vert D u\Vert (\Omega)=\sup\left\{\int_{\Omega}u \,\text{div} (v)\, dx\,:\, v\in C^{\infty}_{0}(\Omega;\R^n) ,\, |v|\leq 1\right\}
\]
is finite. 
As observed in \cite{GBR2021}, the above characterization of $W^{1,1}$-extension domains holds only in the plane. This is essentially because  the planar topology  allows one to write the  essential boundary of a set of finite perimeter as the union of Jordan loops, see \cite[Corollary 1]{ACMM2001} (recalled in Proposition \ref{prop:perimeter_decomposition} below).

In higher dimension where such decomposition result does not hold, the characterization is written in terms of sets of finite perimeter. Before going to this characterization, let us recall an earlier result
on $BV_l$-extension domains, where
\[
BV_l(\Omega)=\{u\in L^{1}_{loc}(\Omega):\, \|D u\|(\Omega)<\infty\}.
\]
In \cite{BM1967}, Burago and  Maz'ya proved  the following characterization of  $BV_l$-extension domains:  $\Omega\subset \R^n$ is a  $BV_l$-extension domain if and only if there exists some constant $C>0$ so that any set $A\subset\Omega$ of finite perimeter in $\Omega$ admits an extension $\widetilde A\subset\R^n$ satisfying $\widetilde A\cap\Omega=A$ and 
\[
P(\widetilde A,\R^n)\leq C P(A,\Omega).
\]
Since $L^{1,1}$-extension domains are known to be $BV_l$-extension domains (the proof of this fact follows the same ideas as one may find in \cite[Lemma 2.4]{KMS2010}), the above property about extension of sets of finite perimeter is a necessary condition both for $BV_l$- and $L^{1,1}$-extension domains.

 In order to turn this into a characterization of $L^{1,1}$- or $W^{1,1}$-extension domains, we have to account for the intersection of the boundary of the extended set with the boundary of the domain, analogously to \eqref{eq:curvep1}.
 This leads to the following characterization in terms of strong extension of sets of finite perimeter \cite{GBR2021}: A bounded domain $\Omega$ is a $W^{1,1}$-extension domain if and only if any set $A\subset\Omega$ of finite perimeter in $\Omega$ admits an extension $\widetilde A\subset\R^n$ satisfying $\widetilde A\cap\Omega=A$, 
\[
P(\widetilde A,\R^n)\leq C P(A,\Omega)\;\;\text{and also}\;\; \mathcal{H}^{n-1}(\partial^M \widetilde A\cap \partial \Omega)=0,
\]
where $\partial^M \widetilde A$ denotes the measure theoretic boundary of $\widetilde A$.
In order to remind ourselves of the analogous condition in the planar simply connected case as the limit of \eqref{eq:characterization_1<p<2}, we can rewrite this in an integral form
\[
\int_{\partial^M \widetilde A} \frac{1}{\chi_{\partial\Omega}(z)} \,d\mathcal H^{n-1}(z) \le C\int_{\Omega\cap \partial^M A} \frac{1}{\chi_{\partial\Omega}(z)} \,d\mathcal H^{n-1}(z).
\]
This motivates the formulation of the following main theorem of this paper.

\begin{theorem}\label{thm:necessity}
Let $\Omega\subset\R^n$ be an $L^{1,p}$-extension domain for some $1 < p<2$. Then for any $\varepsilon >0$ and any measurable set $A \subset \Omega$ 
there exists a set $\widetilde A \subset \mathbb R^n$   with $A = \tilde A \cap \Omega$ and 
\begin{equation}\label{eq:maininequality}
\int_{\partial^M \widetilde A} \dist(z,\partial \Omega)^{1-p} \,d\mathcal H^{n-1}(z) \le C(n,p,\varepsilon)\|E\|^{n+p+\varepsilon}\int_{\Omega\cap \partial^M A} \dist(z,\partial \Omega)^{1-p} \,d\mathcal H^{n-1}(z),
\end{equation}
where $\|E\|$ denotes the norm of the $L^{1,p}$-extension operator, and the constant $C(n,p,\varepsilon)$ depends only on $n$, $p$ and $\varepsilon$.
\end{theorem}


Let us immediately comment on the range $1 < p < 2$ of exponents and the use of the homogeneous Sobolev space $L^{1,p}$ in Theorem \ref{thm:necessity}.
The reason for the range of exponents is that if $p \ge 2$, then the integral on the right-hand side of \eqref{eq:maininequality} is infinite for any set $A$ 
for which the zero-extension $\widetilde A = A$ would not satisfy \eqref{eq:maininequality}.
Thus, for $p \ge 2$ the conclusion of Theorem \ref{thm:necessity} provides no information. 

The use of the homogeneous Sobolev space is natural for scaling invariant results. 
In the case $\Omega$ is bounded, the result still applies for $W^{1,p}$-extension domains because these are  known to be $L^{1,p}$-extension domains as well (see \cite{K1990}). When thinking about moving between $W^{1,p}$- and $L^{1,p}$-extensions in bounded domains, one should observe that for a set $A$ occupying most of $\Omega$ (in our proof, for $A$ satisfying $|A| > \frac12|\Omega|$) the extension $\widetilde A$ satisfying \eqref{eq:maininequality} has to contain all of the space  $\mathbb R^n$ that is sufficiently far away from $\Omega$.

It is worth noticing also that, if $\Omega$ is bounded, any measurable set $A\subset\Omega$ for which the right hand side of the inequality \eqref{eq:maininequality} is finite must be of finite perimeter in $\Omega$, and also the set
$\widetilde A$ that we construct will be of finite perimeter in $\R^n$. If $\Omega$ were unbounded we would only have that $A$ and $\widetilde A$ are locally of finite perimeter in $\Omega$ and in $\R^n$, respectively.

One might wonder if the condition in Theorem \ref{thm:necessity} is also sufficient for $\Omega$ to be an $L^{1,p}$-extension domain. It turns out that this is not the case: Suppose $\Omega \subset \mathbb R^n$ is an arbitrary domain. We can modify $\Omega$ to a new domain $\Omega' = \Omega \setminus \bigcup_{i=1}^\infty B(x_i,r_i)$, where the balls $B(x_i,r_i) \subset \Omega$ are selected in such a way that $B(x_i,2r_i)\setminus B(x_i,r_i) \subset \Omega'$ (giving that we have an extension operator from $L^{1,p}(\Omega')$ to $L^{1,p}(\Omega)$), but so that they accumulate densely enough to the boundary of $\Omega$ so that for any $A \subset \Omega'$ with the right-hand side of \eqref{eq:maininequality} finite for $\Omega'$, the extension $\tilde A$ can be taken to be zero outside $\Omega$ so that the condition 
\eqref{eq:maininequality} again gives us no information on $\Omega$.

In dimensions at least three, one can make the above idea into a construction of a topologically nice extension domain with large boundary. In the version of the construction that we use to prove the following theorem, the removed balls from the domain are replaced by removed tubes, and they accumulate only to a large portion of the boundary instead of the whole boundary. 
\begin{theorem}\label{thm:example}
There exists a domain $\Omega\subset \R^3$ such that $\Omega = h(B(0,1))$ for a homeomorphism $h\colon \R^3 \to \mathbb \R^3$, $\dim_{\mathcal H}(\partial\Omega) = 3$ and $\Omega$ is a $W^{1,p}$-extension domain for all $p \in [1,\infty]$. 
\end{theorem}

 Note that the domain in Theorem \ref{thm:example} cannot be  an $(\varepsilon,\delta)$-domain, nor a John domain, since these domains have porous boundaries and hence their Hausdorff (and packing) dimensions would be strictly less than three. We also reiterate that the same type of example is not possible in $\R^2$ by the dimension bounds on the boundary of a simply connected planar Sobolev extension domain given in \cite{LRT2020}.

 We wrote the dependence on the norm of the extension operator explicitly in Theorem \ref{thm:necessity} mainly in order to start the investigation of the dependence between the norm and the constant $C$ in \eqref{eq:characterization_1<p<2}. Using this explicit form, we obtain a more quantified version of the necessity of \eqref{eq:characterization_1<p<2}.

\begin{theorem}\label{thm:planarcurveintro}
Let $\Omega\subset\R^2$ a bounded simply connected $L^{1,p}$-extension domain for some $1<2<p$. Then for every $\varepsilon>0$ there exists a constant $C(p,\varepsilon)>0$ such that for all $z_1,z_2\in \partial \Omega$ there exists a curve $\gamma\subset\R^2\setminus \Omega$ joining $z_1$ and $z_2$ so that 
\begin{equation}\label{eq:curvecond}
\int_\gamma \dist^{1-p}(z,\partial \Omega)\, \d s(z)\leq 
C(p,\varepsilon)\|E\|^{\frac{4+4p-p^2}{2-p}+\varepsilon}|z_1-z_2|^{2-p}.  
\end{equation}
\end{theorem}
We do not claim nor expect the dependence on $\|E\|$ in \eqref{eq:curvecond} to be sharp. However, our proof of Theorem \ref{thm:planarcurveintro} written in Section \ref{sec:quantitativecurve} gives the first explicit dependence. Since \eqref{eq:characterization_1<p<2} is a characterization, one could also try to get the dependence of the operator norm $\|E\|$ on the curve condition constant $C$. This direction of the proof of the characterization in \cite{KRZ2015} is more technical. Consequently, we suspect the quantitative dependence in this direction to be more difficult to obtain.

\subsection*{Acknowledgements}
The authors acknowledge the support from the Academy of Finland, grant no. 314789. This work was partly done while the first-named author was enjoying a postdoctoral position at the Department of Mathematics and Statistics of the University of Jyv\"akyl\"a. He also wants to thank the department for their kind hospitality during his time there.


\section{Preliminaries}

In what follows, we use the notation $C(\cdot)$ to mean a strictly positive and finite function on the parameters listed in the parentheses, i.e. a constant once the listed parameters are fixed. The function (constant) may change between appearances even within a chain of inequalities.

 For any point $x \in \mathbb R^n$ and radius $r>0$ we denote the open ball by $
 B(x,r) = \{y \in \mathbb R^n\,:\, |x-y| < r\}.$ More generally, for a set $A \subset \mathbb R^n$ we define the open $r$-neighbourhood as
 \[
 B(A,r) = \bigcup_{x \in A} B(x,r).
 \]
 
 We denote by $|A|$ the $n$-dimensional outer Lebesgue measure of a set $A \subset \mathbb R^n$.
 For any Lebesgue measurable subsets $A \subset \Omega\subset\R^n$ and any point $x \in \mathbb R^n$ we define the upper density of $A$ at $x$ over $\Omega$ as
 \[
 \overline{D}(A,\Omega,x) = \limsup_{r\searrow 0}\frac{|A\cap B(x,r)|}{|B(x,r)\cap\Omega|},
 \]
 and the lower density of $A$ at $x$ over $\Omega$ as
 \[
 \underline{D}(A,\Omega,x) = \liminf_{r\searrow 0}\frac{|A\cap B(x,r)|}{|B(x,r)\cap \Omega|}.
 \]
 If  $\overline{D}(A,\Omega,x) = \underline{D}(A,\Omega,x)$, we call the common value the density of $A$ at $x$ over $\Omega$ and denote it by $D(A,\Omega,x)$.
  If $\Omega=\R^n$ we simply write $\underline{D}(A,x),\overline{D}(A,x)$, and $D(A,x)$.
 The essential interior of $A$ 
 is then defined as
 \[
 \mathring{A}^M = \{x \in \mathbb R^n \,:\, D(A,x)=1 \},
 \]
 the essential closure of $A$ 
 as
 \[
 \overline{A}^M = \{x \in \mathbb R^n \,:\, \overline{D}(A,x) > 0\},
 \]
 and the essential boundary of $A$ 
 as
 \[
 \partial^M A = \{x \in \mathbb R^n\,:\, \overline{D}(A,x) > 0\text{ and } \overline{D}(\mathbb R^n \setminus A,x) > 0\}.
 \]

As usual,  $\mathcal{H}^s(A)$ stands for the $s$-dimensional Hausdorff measure of a set $A \subset \R^n$ obtained as the limit
\[
\mathcal H^s(A) = \lim_{\delta \searrow 0}\mathcal H_\delta^s(A),
\]
where $\mathcal H_\delta^s(A)$ is the $s$-dimensional Hausdorff $\delta$-content of $A$ defined as
\[
\mathcal H_\delta^s(A) = \inf\left\{\sum_{i=1}^\infty \diam(U_i)^s \,:\, A \subset \bigcup_{i=1}^\infty U_i, \diam(U_i) \le \delta\right\}.
\]

By a dyadic cube we refer to $Q = [0,2^{-k}]^n + \mathtt {j} \subset \R^n$ for some $k \in \mathbb Z$ and $\mathtt{j} \in 2^{-k}{\mathbb Z}^n$. We denote the side-length of such dyadic cube $Q$ by $\ell(Q) := 2^{-k}$.  

\subsection{Sets of finite perimeter}
A Lebesgue measurable subset $A\subset \R^n$ has finite perimeter in an open set  $\Omega$ if $\chi_A\in BV(\Omega)$, where $\chi_A$ denotes the characteristic function of the set $A$. We set $P(A,\Omega)=\|D \chi_A\|(\Omega)$ and call it the perimeter of $A$ in $\Omega$. Here 
\[
\Vert D \chi_A\Vert (\Omega)=\sup\left\{\int_{A}\,\text{div} (v)\, dx:\, v\in C^{\infty}_{0}(\Omega;\R^n) ,\, |v|\leq 1\right\}
\]
denotes the total variation of $\chi_A$ on $\Omega$.

It is well known that   a set $E$ has finite perimeter in $\Omega$ if and only if $ \mathcal{H}^{n-1}(\partial ^M E\cap\Omega)<\infty$ (see \cite[Section 4..5.11]{F1969}). Let us recall as well the isoperimetric inequality, which follows from the  $(1^*,1)$-Poincar\'e inequality for $BV$ functions (see for instance \cite[Theorem 3.44]{AFP2000}).

\begin{proposition}\label{iso.ine}
Let $\Omega\subset\R^n$ be an open set and $A\subset\Omega$ a set of finite perimeter in  $\Omega$. Let also $Q,Q'\subset\Omega$ be two dyadic cubes  with $\frac{1}{4}\ell(Q')\leq \ell(Q)\leq 4\ell(Q')$ and so that  $\interior(Q\cup Q)$ is connected.   Then we have 
\begin{equation}\label{iso_cubes}
P(A,\interior(Q\cup Q'))\geq C(n) \min\{|A\cap (Q\cup Q')|^{1-1/n}, | (Q\cup Q')\setminus A|^{1-1/n}\} .    
\end{equation}

Moreover, for every $r>0$ and $x\in\overline\Omega$,
\begin{equation}\label{iso_ball}
P(A,B(x,r)\cap\Omega)\geq C(n)\min\{ |A\cap B(x,r)|^{1-1/n}, | (B(x,r)\cap\Omega)\setminus A)|^{1-1/n}  \} .
\end{equation}
\end{proposition}

The study of the boundary of sets of finite perimeter can be reduced to the study of Jordan loops via the following decomposition result from  \cite[Corollary 1]{ACMM2001}.
\begin{proposition} \label{prop:perimeter_decomposition}
 Let $E \subset \R^2$ have finite perimeter. Then, there exists a unique decomposition of $\partial^M E$ into rectifiable Jordan curves $\{C_i^+, C_k^- : i,k \in \N \} $, modulo $\mathcal{H}^1$-measure zero sets, such that 
 \begin{enumerate}
     \item Given $\interior(C_i^+)$, $\interior (C_k^+)$, $i \neq k$ they are either disjoint or one is contained in the other; given $\interior (C_i^-)$, $\interior (C_k^-)$, $i \neq k$, they are either disjoint or one is contained in the other. Each $\interior (C_i^-)$ is contained in one of the $\interior (C_k^+)$.
     \item $P(E,\R^2) = \sum_i \mathcal{H}^1 ( C_i^+) + \int _k \mathcal{H}^1 (C_k^-)$.
     \item If $\interior ( C_i^+) \subset \interior (C_j^+)$, $i \neq j$, then there is some rectifiable Jordan curve $C_k^-$ such that $\interior (C_i^+) \subset \interior (C_k^-) \subset \interior (C_j^+)$. Similarly, if $\interior(C_i^-) \subset \int(C_j^-)$, $i \neq j$, then there is some rectifiable Jordan curve $C_k^+$ such that $\interior (C_i^-) \subset \interior (C_k^+) \subset \interior (C_j^-)$.
     \item Setting $L_j = \{ i : \interior (C_i^-) \subset \interior (C_j^+) \}$, the sets $Y_j = \interior (C_j^+) \setminus \bigcup_{i \in L_j} \interior(C_i^-)$ are pairwise disjoint, indecomposable and $E = \bigcup_j Y_j$
 \end{enumerate}
\end{proposition}

\subsection{Whitney decomposition}\label{sec:whitney}

If $\Omega\subset \mathbb R^n$ is an open set, not equal to the entire space $\R^n$, we let $\mathcal W = \{Q_i\}_{i=1}^\infty$ be the standard \emph{Whitney decomposition} of $\Omega$, by which we mean that it satisfies the following properties:
\begin{itemize}
    \item[(W1)] Each $Q_i$ is a closed dyadic cube inside $\Omega$.
    \item[(W2)] $\Omega=\bigcup_i Q_i$ and for every $i \ne j$ we have $\text{int}(Q_i) \cap \text{int}(Q_j) = \emptyset$.
    \item[(W3)] For every $i$ we have $\ell(Q_i) \le \dist(Q_i,\partial \Omega)\le 4\sqrt{n} \ell(Q_i)$,
    \item[(W4)]  If $Q_i\cap Q_j \ne \emptyset$, we have $\frac{1}{4}\ell(Q_i) \le  \ell(Q_j)\le 4\ell(Q_i)$.
\end{itemize}
The reader can find a proof of the existence of such a dyadic decomposition of the set $\Omega$ in \cite[Chapter VI]{stein}.

For such Whitney decomposition $\mathcal W$ we take a partition of unity $\{\psi_i\}_{i=1}^\infty$ so that for every $i$ we have $\psi_i \in C^{\infty}(\mathbb R^n)$, {$\text{spt}(\psi_i)=\{x\in\R^n:\,\psi_i(x)\neq 0\} \subset B(Q_i,\frac{1}{16}\ell(Q_i))$,} $\psi_i \ge 0$, $|\nabla \psi_i| \le C(n) \ell(Q_i)^{-1}$, and 
\[
\sum_{i=1}^\infty \psi_i = \chi_\Omega.
\]
Notice that for each $Q_i \in \mathcal W$ the above together with the bound on the size of the supports and (W4) implies
\begin{equation}\label{eq:12Q}
 \psi_i(x) = 1 - \sum_{j \ne i}\psi_j(x) = 1 \qquad \text{for all }x \in \frac12Q_i.
\end{equation}
In order to ease the notation, we denote for each $Q_i \in \mathcal W$ by $\mathcal N(Q_i)$  the collection of neighboring cubes that have a common face with $Q_i$:
\[
\mathcal N(Q_i) = \left\{Q_j \in \mathcal W\setminus \{Q_i\}\,:\, \text{int}(Q_i\cup Q_j) \text{ is connected}\right\}.
\]

\subsection{Size estimates}
In this subsection we recall the remaining key auxiliary results that will be used in the paper.

The following lemma is a modification of \cite[Lemma 3.2]{KRZ2015}. This version of the estimate was proven in \cite[Lemma 2.3]{GBRT2021}. (Here we can simplify the presentation a bit since we do not need an exceptional set $F$.)

\begin{proposition}\label{lma:cube0to1}
Let $Q$ be an $n$-dimensional cube in $\mathbb{R}^n$ with sides parallel to the coordinate axes. Let $f\in C(Q)\cap W^{1,p}(\mathbb{R}^n)$ for some $1\leq p< \infty$ and suppose there exists $\delta\in (0,1)$ so that

\[
\min\left(|\{y \in Q\,:\, f(y) \le 0\}|, |\{y \in Q\,:\, f(y) \ge 1\}|\right) > \delta \ell(Q)^n.
\]
Then
\[
\int_{Q}|\nabla f(y)|^p\,\d y \ge C(n,p)\delta^{\frac{n-p}{n}} \ell(Q)^{n-p}.
\]
\end{proposition}

For $L^{1,p}$-extension domains $\Omega$ with $1\leq p<\infty$ the following measure density condition holds for points $x\in\overline \Omega$. This version of the measure density condition was proven in \cite[Proposition 2.2]{GBRT2021} following the results in \cite{HKT2008}, see also \cite{K1990}.

\begin{proposition}\label{thm:mdc}
Let $1\leq p<\infty$ and let  $\Omega \subset \R^n$ be a Sobolev $L^{1,p}$-extension domain with an extension operator $E$. 
Then, for all $x \in \overline\Omega$ and 
\[
r\in \left(0,\min \left\lbrace 1, \left(\frac{|\Omega|}{2\, |B(0,1)|}\right)^{1/n}\right\rbrace \right),
\]
we have 
\[
 |\Omega\cap B(x,r)| \ge C(n,p)\|E\|^{-n}r^n.
\]
\end{proposition}


\section{Proof of the necessary condition}\label{sec:necessity}

In this section we prove Theorem \ref{thm:necessity}. In order to make the structure of the proof clearer, we first present the proof assuming the more technical parts proven. These technical parts are stated as separate lemmata. They are then proven after the proof of Theorem \ref{thm:necessity}.

\begin{proof}[Proof of Theorem \ref{thm:necessity}]
We start with a measurable set $A\subset\Omega$ so that 
\begin{equation}\label{eq:int_finite}
\int_{\Omega \cap \partial^M A} \dist(z,\partial \Omega)^{1-p} \,d\mathcal H^{n-1}(z) < \infty.
\end{equation}
Notice that if \eqref{eq:int_finite} fails, we can simply take $\tilde A = A$ as the set satisfying the required inequality \eqref{eq:maininequality}.


Following the definitions in Subsection \ref{sec:whitney},
let $\mathcal W=\{Q_i\}$ and $\widetilde{\mathcal W}=\{\widetilde{Q}_i\}$ be the Whitney decompositions of $\Omega$ and $\R^n\setminus\overline\Omega$ respectively, and let $\{\psi_i\}_{i=1}^\infty$ be the partition of unity in $\Omega$ subordinate to $\mathcal W=\{Q_i\}$.

 We first modify our set $A$  by means of selecting those Whitney cubes that intersect the set $A$ in a large enough measure set. Namely, we let
\[
A'=\bigcup_{\substack{ Q_i\in \mathcal{W}\\ |A\cap Q_i|>\frac{1}{2}|Q_i|}} Q_i.
\]
It will be easier to handle this new set $A'$ rather than the original set $A$.

Next, for the constant $c=20\sqrt{n}$ we define
\[
A_0=\bigcup_{\substack{\widetilde Q\in\widetilde{\mathcal W}\\ |c\widetilde Q\cap A'|>|c\widetilde Q\cap (\Omega \setminus A')| }}\widetilde Q  .
\]
Our extension of the set $A$ is then defined as 
\[
\widetilde A=A\cup A_0.
\]

The task in proving Theorem \ref{thm:necessity} is now to show that the choice of $\widetilde A$ above works.
We divide this task into several lemmata.
The first lemma justifies the replacement of $A$ by $A'$.

\begin{lemma}\label{lem:moving_to_A'}
For the sets $A$ and $A'$ above we have
\[
\int_{\Omega \cap \partial^M A'} \dist(z,\partial \Omega)^{1-p} \,d\mathcal H^{n-1}(z)\le C(n) \int_{\Omega \cap \partial^M A} \dist(z,\partial \Omega)^{1-p} \,d\mathcal H^{n-1}(z).
\]
\end{lemma}

The next step is then to go from the set $A'$ to a Sobolev function to which we can apply our $L^{1,p}$-extension operator.
This is done with a Whitney smoothing operator $S_{\mathcal W}$ defined via the partition of unity $\{\psi_i\}_{i=1}^\infty$ for $\Omega$.
We define for any $v \in L^{1}_{loc}(\Omega)$ a smoothened version of $v$ as
\begin{equation}\label{eq:SWdef}
 (S_{\mathcal W}v)(x) = \sum_{i=1}^\infty \psi_i(x)\frac{1}{|Q_i|} \int_{Q_i} v(y)\,d(y).
\end{equation}
Whitney smoothing operators similar to the one above have been used for instance in 
\cite{HK1998,BHS2002,LLW2020,GBR2021}.

In addition to smoothing the function, the operator $S_{\mathcal W}$ has the important property of leaving the trace of the function unmodified on the boundary of $\Omega$. Within our proof, this is the content of the last Lemma \ref{lem_strong_ext.}.
The second lemma relates the integral in \eqref{eq:maininequality} to the $L^p$-norm of the gradient of the smoothened version of the indicator function. We write the lemma for a general set $F$, but here inside the proof of Theorem \ref{thm:necessity} use it only for the set $A'$. 

\begin{lemma}\label{lem:1}
 Let $S_{\mathcal W}$ be the operator defined in \eqref{eq:SWdef}. Then for any measurable $F \subset \Omega$ with
 \[
\int_{\Omega \cap \partial^M F} \dist(z,\partial \Omega)^{1-p} \,d\mathcal H^{n-1}(z) < \infty
\]
 we have $S_{\mathcal W}\chi_F \in C^\infty(\Omega)$ and 
 \[
 \|\nabla S_{\mathcal W}\chi_F\|_{L^p(\Omega)}^p \le C(n,p) \int_{\Omega \cap \partial^M F} \dist(z,\partial \Omega)^{1-p} \,d\mathcal H^{n-1}(z).
 \]
\end{lemma}

We now use $S_{\mathcal W}$ to pass from the characteristic function $\chi_{A'}$ to a Sobolev function
\[
u=S_{\mathcal W}\chi_{A'}\in L^{1,p}(\Omega).
\]
Lemma \ref{lem:1} together with Lemma \ref{lem:moving_to_A'} then gives us
\begin{equation}\label{eq:main_thm_1}
    \|\nabla u\|_{L^p(\Omega)}^p \le C(n) \int_{\Omega \cap \partial^M A} \dist(z,\partial \Omega)^{1-p} \,d\mathcal H^{n-1}(z).
\end{equation}

The third lemma shows that the extension $\widetilde A$ of the set $A$ has the correct property outside the closure of the domain $\Omega$. The fact that $\Omega$ is a Sobolev-extension domain is used in the proof of this lemma. Recall that $\|E\|$ denotes the norm of the $L^{1,p}$-extension operator. 

\begin{lemma}\label{lem:extended_set}
 With the $A_0$ and $u$ defined above, for every $\varepsilon >0$ we have
 \begin{equation*}\label{thm:last_step}
\int_{\partial^M A_0\setminus \overline\Omega} \dist(z,\partial \Omega)^{1-p} \,d\mathcal H^{n-1}(z) \le C(n,p,\varepsilon)\|E\|^{n+p+\varepsilon} \|\nabla u\|^{p}_{L^p(\Omega)}.
\end{equation*}
\end{lemma}

Now, the combination of Lemma \ref{lem:extended_set} and the inequality \eqref{eq:main_thm_1} gives 
\begin{equation}\label{eq:main_thm_2}
\int_{\partial^M A_0\setminus \overline\Omega} \dist(z,\partial \Omega)^{1-p} \,d\mathcal H^{n-1}(z) \le C(n,p,\varepsilon)\|E\|^{n+p+\varepsilon}
\int_{\Omega \cap \partial^M A} \dist(z,\partial \Omega)^{1-p} \,d\mathcal H^{n-1}(z).
\end{equation}

The last lemma deals with the boundary of $\Omega$, where in principle some part of the measure theoretic boundary of $\widetilde A$ could live and cause the integral on the left-hand side of \eqref{eq:maininequality} to be infinite.

\begin{lemma}\label{lem_strong_ext.}
With our set $\widetilde A$ defined above, we have
$\mathcal{H}^{n-p}(\partial^M \widetilde A\cap\partial\Omega)=0$.
\end{lemma}

Since we can write
\begin{equation*}
\partial^M \widetilde A=(\partial^M A_0 \setminus \overline \Omega)\cup(\Omega \cap \partial^M A) \cup  (\partial^M \widetilde A\cap\partial\Omega),
\end{equation*}
we can split the integral on the left-hand side of \eqref{eq:maininequality} and use the estimate \eqref{eq:main_thm_2}
and Lemma \ref{lem_strong_ext.} to obtain
\begin{align*}
\int_{\partial^M \widetilde A} \dist(z,\partial \Omega)^{1-p} \,d\mathcal H^{n-1}(z) & = \int_{\partial^M A_0\setminus \overline\Omega} \dist(z,\partial \Omega)^{1-p} \,d\mathcal H^{n-1}(z) \\
& \quad + \int_{\Omega \cap \partial^M A} \dist(z,\partial \Omega)^{1-p} \,d\mathcal H^{n-1}(z)\\
& \quad +\int_{\partial^M \widetilde A\cap\partial\Omega} \dist(z,\partial \Omega)^{1-p} \,d\mathcal H^{n-1}(z)\\
& \le C(n,p,\varepsilon)\|E\|^{n+p+\varepsilon}
\int_{\Omega \cap \partial^M A} \dist(z,\partial \Omega)^{1-p} \,d\mathcal H^{n-1}(z)\\
& \quad + \int_{\Omega \cap \partial^M A} \dist(z,\partial \Omega)^{1-p} \,d\mathcal H^{n-1}(z) + 0\\
& \le C(n,p,\varepsilon)\|E\|^{n+p+\varepsilon}
\int_{\Omega \cap \partial^M A} \dist(z,\partial \Omega)^{1-p} \,d\mathcal H^{n-1}(z).
\end{align*}
Thus, we conclude the proof of Theorem \ref{thm:necessity}.
\end{proof}

Let us then focus on proving the lemmata we used in the proof of Theorem \ref{thm:necessity}.

\begin{proof}[Proof of Lemma \ref{lem:moving_to_A'}]
Setting $a_i=\frac{|A'\cap Q_i|}{|Q_i|}\in \{0,1\} $ we start by writing
\begin{align*}
\int_{\Omega \cap \partial^M A'} \dist(z,\partial \Omega)^{1-p} \,d\mathcal H^{n-1}(z)
& \leq\sum_{Q_i}\int_{Q_i \cap \partial^M A'} \dist(z,\partial \Omega)^{1-p} \,d\mathcal H^{n-1}(z)\\
&\leq \sum_{Q_i} \sum_{Q_j \in \mathcal N(Q_i)} \ell(Q_i)^{1-p} P(A',Q_i\cup Q_j)\\
&= \sum_{Q_i} \sum_{\substack{Q_j \in \mathcal N(Q_i)\\ a_i\neq a_j}} \ell(Q_i)^{1-p} P(A',Q_i\cup Q_j),
\end{align*}
and 
\begin{align*}
\int_{\Omega \cap \partial^M A} \dist(z,\partial \Omega)^{1-p} \,d\mathcal H^{n-1}(z)
& \ge \frac12 \sum_{Q_i}\int_{Q_i \cap \partial^M A} \dist(z,\partial \Omega)^{1-p} \,d\mathcal H^{n-1}(z)\\
& \ge C(n,p) \sum_{Q_i} \sum_{Q_j \in \mathcal N(Q_i)} \ell(Q_i)^{1-p} P(A,Q_i\cup Q_j) \\
& \ge C(n,p) \sum_{Q_i} \sum_{\substack{Q_j \in \mathcal N(Q_i)\\a_i\neq a_j}} \ell(Q_i)^{1-p} P(A,Q_i\cup Q_j).
\end{align*}
Hence, we only need to check that for $i,j\in\N$ with $Q_j \in \mathcal N(Q_i)$ and $a_i\neq a_j$  we have $P(A', Q_i\cup Q_j)\le C(n) P(A, Q_i\cup Q_j)$. Assuming without loss of generality that $a_i=1$ and $a_j=0$, this is seen by using the isoperimetric inequality  \eqref{iso_cubes}
\begin{align*}
P(A, Q_i\cup Q_j)&\ge C(n) \min\{ |A\cap (Q_i\cup Q_j)|^{1-1/n}, |(Q_i\cup Q_j)\setminus A|^{1-1/n}\}\\
&\ge C(n) \min\{|A\cap Q_i|^{1-1/n},|Q_j\setminus A|^{1-1/n}\}\\
&\ge C(n) \min\{(1/2)^{1-1/n} \ell(Q_i)^{n-1}, (1/2)^{1-1/n} \ell(Q_j)^{n-1}\}\\
&\ge C(n)  \ell(Q_i)^{n-1}\\
&\ge C(n) P(A', Q_i\cup Q_j). \qedhere
\end{align*}
\end{proof}

\begin{proof}[Proof of Lemma \ref{lem:1}]
From the definition of $S_{\mathcal W}$, for every $Q_i\in\mathcal W$  we get
\[
\|\nabla S_{\mathcal W}\chi_F\|^{p}_{L^p(Q_i)}\le C(n,p) \sum_{Q_j \in \mathcal N(Q_i)} \ell(Q_j)^{n-p}|a_i-a_j|,
\]
where 
\[
a_i=\dfrac{1}{|Q_i|}\int_{Q_i}\chi_{F}(x)\,dx=\dfrac{|F\cap Q_i|}{|Q_i|}.
\]
Assume that we have $i,j\in\N$ with $Q_j \in \mathcal N(Q_i)$. We may further assume that $a_i\ge a_j$.
Then, by using the isoperimetric inequality  \eqref{iso_cubes} we get
\begin{align*}
P(F, Q_i\cup Q_j)&\ge C(n) \min\{ |F\cap (Q_i\cup Q_j)|^{1-1/n}, |(Q_i\cup Q_j)\setminus F|^{1-1/n}\}\\
&\ge C(n) \min\{|F\cap Q_i|^{1-1/n},|Q_j\setminus F|^{1-1/n}\}\\
&\ge C(n) \min\{(a_i)^{1-1/n} \ell(Q_i)^{n-1}, (1-a_j)^{1-1/n} \ell(Q_j)^{n-1}\}\\
&\ge C(n)  \ell(Q_i)^{n-1}|a_i-a_j|^\frac{n-1}{n}\\
&\ge C(n)  \ell(Q_i)^{n-1}|a_i-a_j|.
\end{align*}

Hence, we have 
\begin{align*}
\|\nabla S_{\mathcal W}\chi_F\|^{p}_{L^p(Q_i)}&\le C(n,p)\sum_{Q_j \in \mathcal N(Q_i)} \ell(Q_j)^{n-p}|a_i-a_j|\\
& \le C(n,p)  \sum_{Q_j \in \mathcal N(Q_i)} \ell(Q_i)^{1-p} P(F,Q_i\cup Q_j). 
\end{align*}
Therefore, by using the finite overlapping between Whitney cubes, we have
\begin{align*}
\|\nabla S_{\mathcal W}\chi_F\|^{p}_{L^p(\Omega)} &
 \le C(n,p) \sum_{Q_i} \sum_{Q_j \in \mathcal N(Q_i)} \ell(Q_i)^{1-p} P(F,Q_i\cup Q_j) \\
& \le C(n,p)\sum_{Q_i}\int_{Q_i \cap \partial^M F} \dist(z,\partial \Omega)^{1-p} \,d\mathcal H^{n-1}(z)\\
& \le C(n,p)\int_{\Omega \cap \partial^M F} \dist(z,\partial \Omega)^{1-p} \,d\mathcal H^{n-1}(z).\qedhere
\end{align*}
\end{proof}

\begin{proof}[Proof of Lemma \ref{lem:extended_set}]
We introduce the following subfamily of Whitney cubes of $\widetilde{\mathcal W}$
\[
\mathcal V_0=\left\{\widetilde Q\in \widetilde{\mathcal W}\, :\, \widetilde Q \subset A_0,  \partial^M A_0\cap \widetilde Q\neq \emptyset\right\}.
\]
We then have
\[
\partial^M A_0\setminus\overline\Omega\subset \bigcup_{\widetilde  Q\in\mathcal V_0}\partial (\widetilde Q).
\]
Let us fix $\widetilde Q\in \mathcal V_0$ for the moment. Then there exists a neighbouring cube $\widetilde Q '\in \widetilde{\mathcal W}$, that is $\widetilde Q'\cap \widetilde Q\neq\emptyset $, so that
$\widetilde Q' \not\subset A_0$.
By the definition of $A_0$, we have
\begin{equation}\label{eq:Qlarge}
 |c\widetilde Q\cap A'|>\frac12|c\widetilde Q\cap \Omega|
\end{equation}
and
\begin{equation}\label{eq:Q'small}
 |c\widetilde Q'\cap (\Omega \setminus A')| \ge \frac12|c\widetilde Q'\cap \Omega|.
\end{equation}
In particular, \eqref{eq:Qlarge} and \eqref{eq:Q'small} imply that
\[
\Omega \not\subset c\widetilde Q \qquad \text{or} \qquad \Omega \not\subset c\widetilde Q'.
\]
Therefore,
\begin{equation}\label{eq:diambound}
\max\{\ell(\widetilde Q), \ell(\widetilde Q')\} \le C(n) \diam(\Omega).
\end{equation}

Combining \eqref{eq:diambound},\eqref{eq:Q'small} and \eqref{eq:Qlarge} with the measure density condition stated in Proposition \ref{thm:mdc}, we get
\[
\min\left\{|c\widetilde Q\cap A'|,|c\widetilde Q'\cap (\Omega \setminus A')|\right\} \ge C(n,p)\|E\|^{-n}\ell(\widetilde Q)^n.
\]
Recall that $u=S_{\mathcal W}\chi_{A'}=\sum_{i=1}^\infty a_i\psi_i $ where $a_i=\frac{|A'\cap Q_i|}{|Q_i|}\in \{0,1\} $. By \eqref{eq:12Q} we have $\psi_i = 1$ on $\frac12Q_i$ and so if $Q\subset A'$, then $u=1$ on $\dfrac{1}{2}Q$ and if $Q\not\subset A'$, then $u=0$ on $\dfrac{1}{2}Q$. 
Therefore,
\[
\min\left\{|\{y \in 9c\widetilde Q\,:\, u(y) \le 0\}|, |\{y \in 9c\widetilde Q\,:\, u(y) \ge 1\}|\right\} > C(n,p) \|E\|^{-n} \ell(9c\widetilde Q)^n.
\]
Let $s \in (1,p)$. Then by Proposition \ref{lma:cube0to1}, we have
 \[
    \left(\int_{9c\widetilde Q} |\nabla Eu(x) |^s \,\d x \right)^\frac{p}{s} \ge \left( C(n,p)\|E\|^{-n}\ell(\widetilde Q)^{n-s}  \right)^\frac{p}{s} \ge C(n,p) \|E\|^{\frac{-np}{s}} \ell(\widetilde Q)^{n - p}\ell(\widetilde Q)^{(\frac{p}{s} - 1)n}.
  \]
 This concludes our estimate for the fixed $\widetilde Q\in \mathcal V_0$.
 
 Now, since $p/s>1$, we may use the boundedness of the Hardy-Littlewood maximal operator 
 \[
 M\colon L^{\frac{p}{s}}(\R^n)\to L^{\frac{p}{s}}(\mathbb R^n),
 \]
 to get 
 \begin{align*}
\int_{\partial^M A_0\setminus \overline\Omega} \dist(z,\partial \Omega)^{1-p} \,d\mathcal H^{n-1}(z)  &\le \sum_{\widetilde Q\in\mathcal V_0} \int_{\partial^M A_0\cap\widetilde Q}  \dist(z,\partial \Omega)^{1-p} \,d\mathcal H^{n-1}(z)\\
&\leq \sum_{\widetilde Q\in\mathcal V_0}\ell(\widetilde Q)^{n-p}\\
&  \le C(n,p) \|E\|^{\frac{np}{s}} \sum_{\widetilde Q\in\mathcal V_0}  \ell(\widetilde Q)^{(1-\frac{p}{s} )n} \left(\int_{9c\widetilde Q} |\nabla Eu(x) |^s \,\d x \right)^\frac{p}{s} \\
  &  \le C(n,p) \|E\|^{\frac{np}{s}}   \sum_{\widetilde Q \in \mathcal V_0} \ell(\widetilde Q)^{n} \left(\bint_{9c \widetilde Q} |\nabla Eu(x) |^s \,\d x \right)^\frac{p}{s} \\
  &  \le C(n,p) \|E\|^{\frac{np}{s}} \sum_{\widetilde Q \in \mathcal V_0}  \int_{\widetilde Q} \left|M(|\nabla Eu |^s)(x)  \right|^\frac{p}{s}\,\d x \\
  &  \le C(n,p) \|E\|^{\frac{np}{s}} \int_{\R^n\setminus \overline{\Omega}} \left|M(|\nabla Eu |^s)(x)  \right|^\frac{p}{s}\,\d x \\
  &  \le C(n,p) \|E\|^{\frac{np}{s}} \int_{\R^n} \left|M(|\nabla Eu |^s)(x)  \right|^\frac{p}{s}\,\d x \\
  & \le C(n,p,s) \|E\|^{\frac{np}{s}} \int_{\R^n} |\nabla Eu(x)|^p \,\d x
  \\
  &\le C(n,p,s) \|E\|^{\frac{np}{s}} ||E||^p \int_{\Omega} |\nabla u(x)|^p \,\d x.
 \end{align*}
 Since we may choose $\frac{p}{s}>1$ to be arbitrarily close to $1$ with the price of enlarging the constant $C(n,p,s)$, the lemma is proven.
 \end{proof}
 


\begin{proof}[Proof of Lemma \ref{lem_strong_ext.}]
We divide the proof into three parts. The parts 1 and 3 will imply the claim of the lemma, while part 2 is needed in the proof of part 3.

\medskip

{\color{blue}Part 1:} For $\mathcal H^{n-p}$-a.e. $x\in\partial \Omega$  the limit  $D(A',\Omega,x)=\lim_{r\to 0}\dfrac{|A'\cap B(x,r)|}{|B(x,r)\cap \Omega|}$ exists and is either $0$ or $1$.

\begin{proof}[Proof of Part 1] 
Let 
\[
F=\left\{x\in\partial\Omega:\, D(A',\Omega,x)\notin\{0,1\} \text{ or the limit does not exist}\right\} 
\]
and assume towards contradiction that $\mathcal{H}^{n-p}(F)>0$. Then, there exists $\delta>0$ so that $\mathcal{H}^{n-p}(F_\delta)>0$ for 
\[
F_{\delta}=\left\{x\in\partial\Omega:\,\exists r^x_i\searrow 0\;\;\text{such that}\; \dfrac{|A'\cap B(x,r^x_i)|}{|B(x,r^x_i)\cap \Omega|}\in[\delta,1-\delta]\right\} .
\]
Fix $\varepsilon\in (0,1)$ and for every $x\in F_\delta$ choose $i$ so that $r^x_i<\varepsilon$, then 
\[
F_\delta\subset\bigcup_{x\in F_\delta} B(x,r^x_i) 
\]
and hence by the Vitali covering theorem (see \cite[Theorem 1.24]{EG2015}) there exists a countable collection $\{B(x_i,r_i)\}_{i\in \N}$ so that 
\begin{equation}\label{lem.measure_controled}
\dfrac{|A'\cap B(x_i,r_i)|}{|B(x_i,r_i)\cap \Omega|}\in[\delta,1-\delta] 
\end{equation}
and $F_\delta\subset \bigcup_{i\in\N}B(x_i,5r_i) $.
Recall that $u = S_{\mathcal W}\chi_{A'}$, and that by Lemma \ref{lem:1} and Lemma \ref{lem:moving_to_A'} we have
\begin{align*}
\|\nabla u\|^{p}_{L^p(\Omega)}& \le C(n)\int_{\Omega \cap \partial^M A'} \dist(z,\partial \Omega)^{1-p} \,d\mathcal H^{n-1}(z)\\
& \le C(n)\int_{\Omega \cap \partial^M A} \dist(z,\partial \Omega)^{1-p} \,d\mathcal H^{n-1}(z) < \infty.
\end{align*}
So, we have  $u\in L^{1,p}(\Omega)$. We extend $u$ to $Eu\in L^{1,p}(\R^n)$.
Observe that for every $i\in\N$, by \eqref{lem.measure_controled} and by the measure density condition (Proposition \ref{thm:mdc}) we have
\[
|A'\cap B(x_i,r_i)|\geq \delta |B(x_i,r_i)\cap\Omega|\geq C(n,p)\|E\|^{-n}\delta r^{n}_{i}
\]
and
\[
|B(x_i,r_i)\setminus A'|\geq |(B(x_i,r_i)\cap\Omega)\setminus A'|\geq \delta |B(x_i,r_i)\cap\Omega|\geq C(n,p)\|E\|^{-n}\delta r^{n}_{i}.
\]
Therefore, by the definition of $u$ via $S_\mathcal W$, and the fact that $A'$ is the union of the same Whitney cubes used in the definition of $S_\mathcal W$, we have 
\[ 
|\{x\in B(x_i,r_i):\, Eu\leq 0\}|\geq C(n,p,\|E\|,\delta)r^{n}_{i}
\]
and
\[
|\{x\in B(x_i,r_i):\, Eu\geq 1\}|\geq C(n,p,\|E\|,\delta)r^{n}_{i}.
\]
Hence, we may apply Proposition \ref{lma:cube0to1} to get the estimate
\[
\int_{B(x_i,r_i)} |\nabla u(y)|^p\,dy\geq C(n,p,\|E\|,\delta)r^{n-p}_{i}.
\]
We can now conclude
\begin{align*}
    \mathcal{H}^{n-p}_{\varepsilon}(F_\delta) & \leq \sum_{i\in\N}(5r_i)^{n-p}\le C(n,p,\|E\|,\delta)  5^{n-p}\sum_{i\in\N} \int_{B(x_i,r_i)} |\nabla u(y)|^p\,dy\\
    & \le C(n,p,\|E\|,\delta) \int_{B(F_{\delta},\varepsilon)}|\nabla u(y)|^p\,dy.
\end{align*}
Using that by the measure density condition $|F_{\delta}|\leq |\partial \Omega|=0 $, the right hand side tends to zero as $\varepsilon\searrow  0$. So $\mathcal{H}^{n-p}(F_\delta)=0$ which is a contradiction. We have thus proven Part 1.
\end{proof}



\medskip

{\color{blue}Part 2:} The following two implications hold for $\mathcal H^{n-p}$-almost every $x \in \partial\Omega$:
\begin{equation}\label{eq:implication1}
\text{If }D(A',\Omega,x)=1\text{, then }D(A,\Omega,x)=1,    
\end{equation}
and
\begin{equation}\label{eq:implication2}
\text{if }D(A',\Omega,x)=0\text{, then }D(A,\Omega,x)=0.
\end{equation}

\begin{proof}[Proof of Part 2]
Let us first show that by going to complements, we only need to prove \eqref{eq:implication1}. Towards this, assume that \eqref{eq:implication1} is true for every measurable set $A \subset \Omega$.
Suppose then that $D(A',\Omega,x)=0$. Call $B=\Omega\setminus A$ and consider the associated 
$$B'=\bigcup_{\left\{ Q_i\in \mathcal{W}:\, |B\cap Q_i|\ge \frac{1}{2}|Q_i| \right\}} Q_i. $$
We have $B'=\Omega\setminus A'$. Since $D(B',\Omega,x)=1$, we have by assumption that $D(B,\Omega,x)=1$. Thus, $D(A,\Omega,x)=0$ and we have shown \eqref{eq:implication2}. (Notice that the form of the definitions of the sets $A'$ and $B'$ differ slightly in that one has a strict inequality while the other does not. However, it is easy to observe that this does not affect the proof below.)

Let us then prove \eqref{eq:implication1}. The argument is similar to the proof of Part 1.
This time we write
\[
G=\left\{x\in\partial\Omega:\, D(A',\Omega,x)= 1 \text{ and }D(A,\Omega,x) \ne 1 \right\} 
\]
and assume towards contradiction that $\mathcal{H}^{n-p}(G)>0$.
 Then, as in the previous proof, there exists $\delta>0$ so that $\mathcal{H}^{n-p}(G_\delta)>0$ for 
\begin{align*}
G_{\delta}=& \left\{x\in\partial\Omega\,:\, \exists r^x_i\searrow 0\;\;\text{such that}\; \dfrac{|A\cap B(x,r^x_i)|}{|B(x,r^x_i)\cap \Omega|}<1-\delta \right.\\
&\qquad \qquad \quad \left.\text{ and } \dfrac{|A'\cap B(x,r)|}{|B(x,r)\cap \Omega|}>\frac12 \text{ for all }0 < r < \delta\right\}.
\end{align*}
Now, at this stage it is enough to notice that by the definition of $A'$ we have
\[
 |A \cap B(x,Cr)| \ge \sum_{\substack{Q_i \in \mathcal W\\ Q_i\subset B(x,Cr)}}|Q_i\cap A| \ge
 \sum_{\substack{Q_i \in \mathcal W\\ Q_i\subset B(x,Cr)}}\frac12 |Q_i\cap A'| \ge \frac12|A'\cap B(x,r)|
\]
so that by the measure density, we have that for some $\delta'>0$ 
\[
G_{\delta} \subset \left\{x\in\partial\Omega:\,\exists r^x_i\searrow 0\;\;\text{such that}\; \dfrac{|A\cap B(x,r^x_i)|}{|B(x,r^x_i)\cap \Omega|}\in[\delta',1-\delta']\right\} .
\]
Now, repeating the proof of Part 1 gives the needed contradiction and proves \eqref{eq:implication1} and thus Part 2.
\end{proof}

\medskip

{\color{blue}Part 3:} 
The following two implications hold for $\mathcal H^{n-p}$-almost  every $x \in \partial\Omega$:
\begin{equation}\label{eq:implication3}
\text{If }D(A',\Omega,x)=1\text{, then }D(\widetilde A,x)=1,    
\end{equation}
and
\begin{equation}\label{eq:implication4}
\text{if }D(A',\Omega,x)=0\text{, then }D(\widetilde A,x)=0.
\end{equation}

\begin{proof}[Proof of Part 3]

Since the definition of $A_0$ passes (up to the difference between a strict and non-strict inequality) to the complements, similarly to the Part 2 it is enough to prove the implication \eqref{eq:implication3}.

Let $x \in \partial \Omega$ with $D(A',\Omega,x)=D(A,\Omega,x)=1$ and $r>0$. (Notice that by Part of the proof, this $D(A,\Omega,x)=1$ holds for $\mathcal H^{n-p}$-almost every $x \in \partial \Omega$ with $D(A',\Omega,x)=1$.)
Now, if $\widetilde Q \in \widetilde{\mathcal W}$ with $\widetilde Q\nsubseteq A_0 $, by the definition of $A_0$ and the measure density condition (Proposition \ref{thm:mdc}) we have
\begin{equation}\label{eq:Qtildebounds}
|c \widetilde Q\cap (\Omega \setminus A')| \ge \frac12 |c \widetilde Q\cap \Omega|
\ge C(n,p,\|E\|)|\widetilde Q|.
\end{equation}
Consider the collection
\[
\mathcal B = \left\{\widetilde Q \in \widetilde{\mathcal W}\,:\, \widetilde Q\nsubseteq A_0,\widetilde Q\cap B(x,r)\neq\emptyset\right\}
\]
and let $x_{\widetilde Q}$ be the center of each $\widetilde Q \in \widetilde{\mathcal W}$.
By the Vitali covering theorem there exists a subcollection $\mathcal B' \subset \mathcal B$ so that
\[
\bigcup_{\widetilde Q \in \mathcal B}B\left(x_{\widetilde Q},\sqrt{n}c\ell(\widetilde Q)\right) \subset \bigcup_{\widetilde Q \in \mathcal B'}B\left(x_{\widetilde Q},5\sqrt{n}c\ell(\widetilde Q)\right)
\]
and
\begin{equation}\label{eq:separateballs}
 B\left(x_{\widetilde Q_1},\sqrt{n}c\ell(\widetilde Q_1)\right) \cap B\left(x_{\widetilde Q_2},\sqrt{n}c\ell(\widetilde Q_2)\right) = \emptyset
\end{equation}
for any two $\widetilde Q_1,\widetilde Q_2 \in \mathcal B'$ with $\widetilde Q_1 \ne \widetilde Q_2$. Notice that
\eqref{eq:separateballs} implies that also 
\[
c\widetilde Q_1 \cap c\widetilde Q_2 = \emptyset.
\]
Hence, by \eqref{eq:Qtildebounds}
\begin{equation}\label{eq:densities_2}
\begin{split}
|B(x,r)\setminus (A_0\cup\Omega)| &\leq \sum_{\widetilde Q \in \mathcal B'} \left|B\left(x_{\widetilde Q},5\sqrt{n}c\ell(\widetilde Q)\right)\right| 
 \le C(n) \sum_{\widetilde Q \in \mathcal B'}|\widetilde Q|\\
& \leq C(n,p,\|E\|) \sum_{\widetilde Q \in \mathcal B'} |c \widetilde Q\cap (\Omega \setminus A')| \\
&\leq C(n,p,\|E\|)  |B(x, Mr)\cap(\Omega \setminus A')|,
\end{split}
\end{equation}
where $M>0$ is a constant depending only on $n$ so that $c\widetilde Q\subset B(x,Mr)$ for any $\widetilde Q \in \widetilde{\mathcal W}$ with  $\widetilde Q\cap B(x, r) \ne \emptyset$.

With \eqref{eq:densities_2} and the measure density condition we can estimate
\begin{align*}
    \frac{|B(x,r)\cap \widetilde A|}{|B(x,r)|}
    & = 1 - \frac{|B(x,r)\setminus (A_0\cup\Omega)|}{|B(x,r)|} - \frac{|B(x,r)\cap (\Omega \setminus A)|}{|B(x,r)|}\\
    & \ge 1 - C(n,p,\|E\|)\frac{|B(x, Mr)\cap(\Omega \setminus A')|}{|B(x,Mr)|} - \frac{|B(x,r)\cap (\Omega \setminus A)|}{|B(x,r)|}\\
    & \ge 1 - C(n,p,\|E\|)\frac{|B(x, Mr)\cap(\Omega \setminus A')|}{|B(x,Mr)\cap \Omega|} - C(n)\frac{|B(x,r)\cap (\Omega \setminus A)|}{|B(x,r)\cap \Omega|}\to 1,
\end{align*}
as $r \searrow 0$, since $D(A',\Omega,x)=D(A,\Omega,x)=1$. This proves \eqref{eq:implication3}.
\end{proof}

 
We can now conclude the proof of the lemma by taking $x \in \partial\Omega$ for which the conclusions of Part 1 and Part 3 above hold. 
Part 1 of the proof says that $D(A',\Omega,x)=\lim_{r\to 0}\dfrac{|A'\cap B(x,r)|}{|B(x,r)\cap \Omega|}$ exists and is either $0$ or $1$.
Then by Part 3 of the proof
\[
D(\widetilde A,x) = D(A',\Omega,x) \in \{0,1\}
\]
and hence $x \notin \partial^M\widetilde A$.
\end{proof}




\section{A quantitative version of the curve condition}\label{sec:quantitativecurve}

In the present section we use Theorem \ref{thm:necessity} to prove Theorem \ref{thm:planarcurveintro}. This gives a quantitative version of a result proven in \cite{KRZ2015}.

Theorem \ref{thm:planarcurveintro} states that if
$\Omega\subset\R^2$ is a bounded simply connected $L^{1,p}$-extension domain for some $1<p<2$ with an extension operator $E$, then for every $\varepsilon >0$ there exists a constant $C(p,\varepsilon)>0$ such that for all $z_1,z_2\in \partial \Omega$ there exists a curve $\gamma\subset\R^2\setminus \Omega$ joining $z_1$ and $z_2$ so that
\begin{equation}\label{eq:curve}
\int_\gamma \dist^{1-p}(z,\partial \Omega)\, \d s(z)\leq C(p,\varepsilon)\|E\|^{\frac{4+4p-p^2}{2-p}+\varepsilon}  |z_1-z_2|^{2-p}.
\end{equation}

The curve condition \eqref{eq:curve} was proven in \cite{KRZ2015} to be a characterization of planar bounded simply connected $W^{1,p}$-extension domains for $1 < p < 2$ (a similar characterizing condition for the complement of a bounded finitely connected planar domain for $p>2$ was given in \cite{Shvartsman}). Here we only prove the necessity, but provide a more explicit estimate on the dependence of the operator norm $\|E\|$ in \eqref{eq:curve}.

The proof in \cite{KRZ2015} of the necessity of \eqref{eq:curve} starts by observing that the domain $\Omega$ is $J$-John, by results in \cite[Theorem 6.4]{K1990}, \cite[Theorem 3.4]{gore1990}, and \cite[Theorem 4.5]{NV1991}. 
Recall that a bounded domain $\Omega \subset \R^2$ is called $J$-John for some constant $J \ge 1$ if there is a point $x_0\in\Omega$ and a constant $J \geq 1$  so that given $z\in\partial \Omega$ we can find a curve parameterized by arc length  $\gamma\subset\Omega$ joining $z$  with $x_0$ so that 
\begin{equation} \label{eq:john}
    \dist(\gamma(t),\partial\Omega)\geq \frac{t}{J}.
\end{equation}

The proof in \cite{KRZ2015} then continues by making a test function in $\Omega$ and by constructing the required curve using conformal maps. These steps make it difficult to track the constants.



The proof of \eqref{eq:curve} in our approach starts by examining two conditions similar to the John condition. We
first prove a quantitative version of the so-called $\text{cig}_d$ condition \eqref{eq:cig_d} (see 
\cite{NV1991} for this and similar conditions) for Sobolev extension domains.
In the lemma below and elsewhere in this section, for an injective curve $\gamma \subset \mathbb R^2$ (possibly defined on an open or half-open interval) and two points $x,y \in \overline{\gamma}$ we denote by $\gamma_{x,y}$ a minimal subcurve of $\gamma$ so that $\gamma_{x,y}\cup\{x,y\}$ is connected.

\begin{lemma}\label{lma:cig_d}
Let $\Omega\subset\R^2$ be a bounded simply connected $L^{1,p}$-extension domain for some $1<p<2$. Then for every $x,y \in \overline{\Omega}$ there exists an injective curve $\gamma \subset \Omega\cup\{x,y\}$ connecting $x$ to $y$ and satisfying 
\begin{equation}\label{eq:cig_d}
 \min\left\{\diam(\gamma_{x,z}),\diam(\gamma_{y,z})\right\} \le
 C_{\text{cig-d}}\dist(z,\partial\Omega)
\end{equation}
for all $z \in \gamma$, where $C_{\text{cig-d}} =C(p)\Vert E \Vert^{\frac{p}{2-p}}$.
\end{lemma}
\begin{proof}
Let us first prove the claim for $x,y \in \partial \Omega$. By the Riemann mapping theorem there exists a conformal map $\varphi \colon \mathbb D \to \Omega$. Since we know that $\Omega$ is a John domain, by \cite[Theorem 2.18]{NV1991} the domain $\Omega$ is finitely connected along its boundary and hence 
$\varphi$ extends as a continuous map to the boundary. We refer to this extension still by $\varphi$. Consider $a \in \varphi^{-1}(\{x\})$ and $b \in \varphi^{-1}(\{y\})$ so that one of the open arcs in $S^{1}$ connecting $a$ and $b$ does not intersect $\varphi^{-1}(\{x,y\})$. Call this arc $I_1$ and write $I_2 = S^1 \setminus (I_1 \cup \{a,b\})$. 

Using the sets $I_1$ and $I_2$ we now define a set
\[
G = \left\{z \in \mathbb D\,:\, \dist_{\Omega,\varphi}(z,I_1) = \dist_{\Omega,\varphi}(z,I_2)\right\},
\]
where the distance $\dist_{\Omega,\varphi}(z,I)$ for a connected set $I \subset S^1$ and a point $z \in \mathbb D$ is defined by
\[
\dist_{\Omega,\varphi}(z,I) = \inf\{\ell(\gamma)\,:\,\gamma \subset \Omega \text{ curve such that }\varphi^{-1}(\gamma) \cup I \cup \{z\} \text{ is connected}\}.
\]
Notice that since $\Omega$ is a John domain we have for any non-empty arc $I$ and any $z \in \mathbb D$ that $\dist_{\Omega,\varphi}(z,I) < \infty$. This can be seen by taking $c \in I$, a sequence $c_i \in \mathbb D$ converging to $c$, the John curves $\gamma_i$ connecting $c_i$ to the John-center $x_0$, and finally a subsequence of $(\gamma_i)$ converging to the desired $\gamma$ giving $\dist_{\Omega,\varphi}(\varphi^{-1}(x_0),I) \le \ell(\gamma) < \infty$. The passage to an arbitrary $z \in \mathbb D$ follows since any two points inside $\Omega$ can be connected by a curve in $\Omega$ of finite length.
Notice moreover, that $\dist_{\Omega,\varphi}(\cdot,I)$ is a continuous function.

We claim that $G \subset \mathbb D$ is a closed set in $\mathbb D$ so that $a$ and $b$ are in the same connected component of $G\cup \{a,b\}$. Suppose this is not the case. Then there exists a path $\alpha$ from $I_1$ to $I_2$ that does not intersect $G$. However, the function 
\[
z \mapsto f(z) = \dist_{\Omega,\varphi}(z,I_1) - \dist_{\Omega,\varphi}(z,I_2)
\]
is continuous in $\mathbb D$, and so in particular along the path $\alpha$. Since $f$ is negative near $I_1$ and positive near $I_2$ the function $f$ must be zero on some point of $\alpha$. This contradicts $G \cap \alpha = \emptyset$ and the claim is proven. Let us call $F$ the connected component of $G \cup \{a,b\}$ that contains the points $a$ and $b$.


Now, consider the following open neighbourhood of $G$
\[
U = \left\{z \in \mathbb D\,:\, \frac12 < \frac{\dist_{\Omega,\varphi}(z,I_1)}{\dist_{\Omega,\varphi}(z,I_2)} < 2 \right\}.
\]
Since $G\cup\{a,b\} \subset U\cup\{a,b\}$ contains a connected component connecting $a$ to $b$, we can find an injective curve $\beta \colon (0,1) \to  U$ so that $\beta \cup \{a,b\}$ is connected. Notice that at this point we do not know if $\beta$ can be extended to $0$ and $1$ as a curve connecting $a$ and $b$,
but after establishing \eqref{eq:beta} below, 
we have that the image curve $\varphi(\beta) \colon (0,1) \to \Omega$ extends uniquely to a curve defined on $[0,1]$ connecting $\varphi(a)$ to $\varphi(b)$.

Next we will show that for any $c \in \beta$ we have
\begin{equation}\label{eq:beta}
    \min\left\{\diam(\varphi(\beta_{a,c})),\diam(\varphi(\beta_{b,c})) \right\} \le C(p)\Vert E \Vert^{\frac{p}{2-p}}\dist(\varphi(c),\partial\Omega).
\end{equation}

Towards proving \eqref{eq:beta}, let $c \in \beta$ and $C>0$ be so that
\begin{equation}\label{eq:betacontradiction}
\min\left\{\diam(\varphi(\beta_{a,c})),\diam(\varphi(\beta_{b,c})) \right\} \ge C\dist(\varphi(c),\partial\Omega).
\end{equation}
The estimate \eqref{eq:beta} is shown if we can prove that necessarily $C \le C(p)\Vert E \Vert^{\frac{p}{2-p}}$. We may assume that $C > 2$.


Let $\gamma^1$ be an injective curve in $\mathbb D \cup \{d_1\}$ joining $d_1 \in I_1$ to $c$ and let $\gamma^2$ be an injective curve in $\mathbb D \cup \{d_2\}$ joining $d_2\in I_2$ to $c$ so that they satisfy
\[
 \ell(\varphi(\gamma^i)) < 2\dist_{\Omega,\varphi}(c,I_i).
\]
Let $c_i \in \gamma^i \cap \beta$ be such that $\gamma_{d_i,c_i}^i \cap \beta$ is a singleton.
Now, the set $\mathbb D \setminus (\gamma_{d_1,c_1}^1 \cup \gamma_{d_2,c_2}^2\cup \beta_{c_1,c_2})$  has two connected components $O_1$ and $O_2$ so that $\beta_{a,c_1} \subset O_1 \cup \{a,c_1\}$ and $\beta_{b,c_2} \subset O_2 \cup \{b,c_2\}$, or with $c_1$ and $c_2$ swapped. Consequently, by \eqref{eq:betacontradiction} the sets
$\Omega_i = \varphi(O_i)$ satisfy
\[
\diam(\Omega^i)  \ge C\dist(\varphi(c),\partial\Omega).
\]

Denote $r = 4\dist(\varphi(c),\partial\Omega)$ and notice that since $c \in U$, we have
\begin{equation}\label{eq:gammaistaybounded}
r  = 4\min\left\{\dist_{\Omega,\varphi}(c,I_1),\dist_{\Omega,\varphi}(c,I_2)\right\} \ge 2\dist_{\Omega,\varphi}(c,I_i)> \ell(\varphi(\gamma^i))\quad \text{for }i=1,2.
\end{equation}
Define the test function
 \[
 u(z) = \chi_{\Omega_1}(z) \max \left\{ \min \left\{ \frac{|\varphi(c)-z| - r}{r},1 \right\},0 \right\}.
 \]
 Clearly $\text{spt} (\nabla u) \subset \overline{B}(\varphi(c), 2r)$ and $|\nabla u| \leq \frac{1}{r}$. Notice, that $u = 0$ on $B(\varphi(c), r)$ and by \eqref{eq:gammaistaybounded} we have $\varphi(\gamma^i) \subset B(\varphi(c), r)$. Hence, for each $z\in  \varphi( \gamma^1 \cup \gamma^2)$
 there exists $\varepsilon>0$ such that $u\equiv 0$ in $B(z,\varepsilon)$. Thus, $u \in W^{1,p}(\Omega)$.

 For the test function $u$ we have
 \[
 \int_{\Omega} |\nabla u|^p \leq 4\pi r^{2-p}.
 \]
Let $E \colon L^{1,p}(\Omega) \to L^{1,p}(\mathbb R^2)$ be the extension operator. Then in polar coordinates
\begin{equation}
\Vert E \Vert ^p 4\pi r^{2-p} \geq \Vert E \Vert^p \int_{\Omega} |\nabla u |^p \geq \int_{\R^2} |\nabla E u|^p \geq \int_{2r}^{Cr} \int_{0}^{2\pi } | \nabla Eu(\alpha,t)|^p t \,\text{d} \alpha\, \text{d} t . \label{eq:btlemma1}
\end{equation}

 By absolute continuity and H\"older's inequality for $2r < t < Cr$ we have
 \[
 1 \leq \int_0^{2\pi } |\nabla Eu(\alpha,t)|t \, \text{d} \alpha \leq \left( \int_0^{2\pi} |\nabla E u(\alpha,t) | ^p t \, \text{d} \alpha \right)^\frac{1}{p} (2\pi t)^{1-\frac{1}{p}}.
 \]

 Hence
 \begin{equation}
 \int_0^{2\pi} |\nabla Eu(\alpha,t)|^p t\, \text{d} \alpha \geq (2\pi t)^{1-p}. \label{eq:btlemma2}
 \end{equation}
 By combining \eqref{eq:btlemma1} and \eqref{eq:btlemma2} we get
 \[
 \Vert E \Vert ^p 4\pi r^{2-p} \geq (2\pi)^{1-p} \int_{2 r}^{Cr} t^{1-p} = \frac{(2\pi)^{1-p}}{ 2-p} \left( (Cr)^{2-p} - (2r)^{2-p} \right).
 \]
 This gives the upper bound 
 \begin{equation}\label{eq:Csimplebound}
 C \leq \left( \Vert E \Vert^p 2^{1+p} \pi^p (2-p) + 2^{2-p} \right)^{\frac{1}{2-p}}.
 \end{equation}
 Thus we have established \eqref{eq:beta} and the lemma is proven in the special case $x,y \in \partial \Omega$.
 
 Let us then consider the general case $x,y \in \overline{\Omega}$. In this case we repeat the previous construction but replace $\Omega$ by the simply connected domain $\Omega' = \Omega\setminus ([x,x']\cup[y,y'])$ where $x',y' \in \partial\Omega$ satisfy
 \[
  |x-x'| = \dist(x,\partial\Omega)\quad \text{and} \quad
  |y-y'| = \dist(y,\partial\Omega)
 \]
 and $[x,x']$ and $[y,y']$ denote the line segments from $x$ to $x'$ and from $y$ to $y'$, respectively.
 Notice that $\Omega'$ is not necessarily a Sobolev extension domain. However, for points near $x$ and $y$ the condition \eqref{eq:beta} is satisfied trivially, and for points far from them, an enlarged ball meets the sets $\varphi(I_1)\setminus ([x,x']\cup[y,y'])$ and
 $\varphi(I_2)\setminus ([x,x']\cup[y,y'])$, so one can still use the argument from the special case.
\end{proof}

The next step is to go from the $\text{cig}_d$ condition \eqref{eq:cig_d} to a $\text{cig}_l$ condition \eqref{eq:cig_l}. Before stating this as a lemma, let us recall the corresponding implication from \cite[p.~385--386]{MS1978} from the so-called $\text{car}_d$ condition to the so-called $\text{car}_l$ condition. This latter condition is very close to the John condition \eqref{eq:john}, where one of the endpoints of all the curves is a fixed point $x_0$.

\begin{lemma}\label{lma:MartioSarvas}
Let $\Omega \subset \R^2$ be a bounded domain and let $0 < \delta \leq 1$. Suppose that there exists a curve
$\gamma \colon [0,1]  \to \Omega$ such that for every $t\in [0,1]$
\[
 \gamma([0,t]) \subset B(\gamma(t), \frac{1}{\delta} \dist(\gamma(t), \partial \Omega)).
\]
Then there exists another arc length parametrized curve 
$\tilde\gamma \colon [0,d] \to \Omega$ with $\tilde\gamma(0)=\gamma(0)$, $\tilde\gamma(d) = \gamma(1)$ and
\[
\dist(\tilde\gamma(t), \partial \Omega) \geq 2^{-14} \delta^2 t \quad \text{for} \quad t\in [0,d].
\]
\end{lemma}

Following the proof of \cite[Theorem 2.14]{NV1991} we now use Lemma \ref{lma:MartioSarvas} to obtain the passage from  $\text{cig}_d$ to $\text{cig}_l$.

\begin{lemma}\label{lma:cig_l}
Let $\Omega\subset\R^2$ be a bounded simply connected domain satisfying the condition \eqref{eq:cig_d} with some constant $C_{\text{cig-d}}$. Then for every $x,y \in \overline{\Omega}$ there exists an injective curve $\gamma \subset \Omega\cup\{x,y\}$ connecting $x$ to $y$ and satisfying 
\begin{equation}\label{eq:cig_l}
 \min\left\{\ell(\gamma_{x,z}),\ell(\gamma_{y,z})\right\} \le
 C_{\text{cig-l}}\dist(z,\partial\Omega)
\end{equation}
for all $z \in \gamma$, where $C_{\text{cig-l}} = 2^{14} C_{\text{cig-d}}^2$.
\end{lemma}
\begin{proof}

Let us first consider the case where $x,y \in \Omega$.
Let $\gamma \subset \Omega$ be a curve joining $x$ and $y$ and satisfying \eqref{eq:cig_d}. Let $x_0 \in \gamma$ be a point such that $\diam(\gamma_{x,x_0}) = \diam(\gamma_{y,x_0})$.
Then, by using Lemma \ref{lma:MartioSarvas} separately to the curves $\alpha^1 = \gamma_{x,x_0}$ and $\alpha^2 = \gamma_{y,x_0}$ there exist arc length parameterized curves $\tilde\alpha^i\colon [0,d_i] \to \Omega$, for $i=1,2$ so that $\tilde\alpha^1(0) = x$, $\tilde\alpha^2(0)=y$, $\tilde\alpha^1(d_1)=\tilde\alpha^2(d_2)=x_0$, and
\[
\dist(\tilde\alpha^i(t), \partial \Omega) \geq 2^{-14} C_{\text{cig-d}}^{-2} t \quad \text{for} \quad t\in [0,d_i] \text{ and }i=1,2.
\]
The concatenation of $\tilde\alpha^1$ and $\tilde\alpha^2$ now gives the curve satisfying \eqref{eq:cig_l}.



Consider then the general case $x,y \in \overline{ \Omega}$.
Let $\{x_i\},\{y_i\} \subset \Omega$ be sequences converging to the points $x$ and $y$, respectively, and let $\gamma^i \colon [0,1] \to \Omega$ be a collection of constant speed parametrized curves connecting $x_i$ to $y_i$, and satisfying \eqref{eq:cig_l} with the same constant $C_{\text{cig-l}}$. Since $\Omega$ is bounded and the lengths of the curves are uniformly bounded, by Arzel\'a-Ascoli there exists a sequence $i_j \nearrow \infty$ and a curve $\gamma$ such that $\gamma_{i_j} \to \gamma$ uniformly. Moreover, by the lower semicontinuity of length
\begin{equation}
\begin{split}
\min \{ \ell (\gamma|_{[0,t]}) ,\ell (\gamma|_{[t,1]})  \} & \leq \liminf_{i\to\infty} \min \{ \ell (\gamma^i|_{[0,t]}) ,\ell (\gamma^i|_{[t,1]})  \}  \\
 & \leq \liminf_{i\to \infty} C_{\text{cig-l}}\dist(\gamma^i(t),\partial\Omega) \leq C_{\text{cig-l}}\dist(\gamma(t),\partial\Omega), \notag
\end{split}
\end{equation}
for all  $t \in [0,1]$. This concludes the general case.

Finally, note that the condition \eqref{eq:cig_l} still holds if we replace $\gamma$ with an injective subcurve \cite[Lemma 3.1]{Falconer86}.
\end{proof}

We are now ready to prove the main result of this section.

\begin{proof}[Proof of Theorem \ref{thm:planarcurveintro}]
Let $1<p<2$, $\varepsilon>0$ and $\Omega \subset \R^2$ a bounded and simply connected $L^{1,p}$-extension domain. Let $z_1,z_2 \in \R^2 \setminus \Omega$, $r = |z_1-z_2|$, $x = \frac{z_1+z_2}{2}$, and denote by $\Omega_i$ the connected components of $B(x,3r)\cap\Omega$ for which $\Omega_i \cap B(x,r) \neq \emptyset$. We divide the proof into two steps. 

{\color{blue}Step 1:} 
Let us first show that we can connect any points $z_1',z_2' \in \partial\Omega_i\cap \partial \Omega$ with a suitable curve.
By Lemmata \ref{lma:cig_d} and \ref{lma:cig_l} we know that there exists a curve $\gamma \colon [0,d] \to \overline{\Omega}$ parametrized by arc length between $z_1'$ and $z_2'$ so that
\begin{equation}\label{eq:proofcig}
\min\left\{\ell(\gamma_{z_1',z}),\ell(\gamma_{z_2',z})\right\} \le
 C_{\text{cig-l}}\dist(z,\partial\Omega)
\end{equation}
for every $z \in \gamma$, where 
\begin{equation}\label{eq:Cigvalue}
    C_{\text{cig-l}} = C(p)\Vert E \Vert^{\frac{2p}{2-p}}.
\end{equation}
Let us denote by $\alpha^j$, $j=1,2$, the subcurves of $\gamma$ such that $\ell(\alpha^1) = \ell(\alpha^2)$ and $\gamma = \alpha^1\cup\alpha^2$. We parametrize $\alpha^j \colon [0,\ell(\alpha^j)] \to \overline{\Omega}$ so that $\alpha^j(0) = z_j'$. We consider two cases. Assume first that $\gamma \subset B(x,4r)$. 
By the condition \eqref{eq:proofcig} we have 
\[
\ell ( \alpha^j|_{[0,t]}) = t \leq  
C_{\text{cig-l}} \dist(\alpha^j(t), \partial \Omega)  \leq 7 C_{\text{cig-l}}r \quad \text{for all } t \in [0,\ell(\alpha^j)].
\]

Let us then consider the case $\gamma \not \subset B(x,4r)$.   Let $\Delta$ be the connected component of $\Omega_i \setminus \gamma$ with $z_1',z_2'\in\partial\Delta$ and let $C$ be a cross-cut in $\Delta$ 
connecting $z_1'$ to a point in $\gamma \setminus B(x,4r)$. Let $w$ be the first point where $C$ intersects $S^1(x,4r)$ when travelling from $z_1'$. 
Denote by $S \subset S^1 (x,4r)$ the maximal arc containing $w$ such that $S\cap (\alpha^1 \cup \alpha^2)= \emptyset$. Let $w_1,w_2$ be the endpoints of $S$. By reordering if necessary, there exist minimal times $t_1,t_2$ such that $w_1 = \alpha^1(t_1)$ and $w_2=\alpha^2(t_2)$.
Let $\alpha$ be the curve parametrizing $\alpha^1|_{[0,t_1]} \cup S \cup \alpha^2 | _{[0,t_2]}$ by arc length. 
By \eqref{eq:proofcig}, we have
\[
\ell(\alpha^j|_{[0,t]}) = t \leq C_{\text{cig-l}} \dist(\alpha^j(t), \partial \Omega) \leq C_{\text{cig-l}} |w_j-z_j| \leq 7C_{\text{cig-l}}r
\]
for all $t \in [0,t_j]$.
Suppose $a\in S$ and $b\in \partial\Omega$ satisfy $|a-b|<\frac12r$. Then, since $S$ is contained in the interior of $\Delta$ and $\Omega_i \subset B(x,3r)$, we have that the line segment $[a,b]$ intersects one of the $\alpha^j$ at some point $\alpha^j(t)$. Since $\ell(\alpha^j|_{[0,t]})>\frac12r$, we have 
\[
|a-b| \ge |b-\alpha^j(t)| \ge \dist(\alpha^j(t), \partial \Omega) \ge \frac{r}{2C_{\text{cig-l}}}.
\]
Consequently,
\[
\dist(S, \partial\Omega) \geq \min \left\{r, \frac{r}{2C_{\text{cig-l}}}\right\} \geq \frac{r}{2C_{\text{cig-l}}}.
\]
Therefore (setting $S= \emptyset$ in the first case),
\begin{equation}\label{eq:intspecialcase}
    \begin{split}
        \int_{\alpha} \dist(z,\partial \Omega)^{1-p} \, \d s (z)& \leq \int_{\alpha^1} \dist(z,\partial \Omega)^{1-p} \, \d s (z) + \int_{\alpha^2} \dist(z,\partial \Omega)^{1-p} \, \d s (z)\\
        & \quad + \int_{S} \dist(z,\partial \Omega)^{1-p} \, \d s (z)  \\
        & \leq \frac{2}{2-p}\frac{1}{C_{\text{cig-l}}^{1-p}} \left(7 C_{\text{cig-l}}r\right)^{2-p} + \left(\frac{r}{2C_{\text{cig-l}}}\right)^{1-p} 8\pi r  \\
        & \leq C(p) C_{\text{cig-l}}|z_1'-z_2'|^{2-p}. 
    \end{split} 
\end{equation}

Finally, a connected component $A$ of $\Omega\setminus \alpha$ has finite perimeter in $\Omega$, and since $\Omega$ is bounded, the extension $\widetilde A$  provided by Theorem \ref{thm:necessity} also has finite perimeter in $\R^2$. By Proposition \ref{prop:perimeter_decomposition}, the boundary $\partial^M \widetilde A$ decomposes into Jordan loops $\{\Gamma_k\}$.

There exists one Jordan curve $\Gamma_k$ in the decomposition with $z_1,z_2\in\Gamma_k$ because the points must be in the same connected component of $\partial^M \widetilde A$.  We now write  $\Gamma_k=\alpha * \tilde{\alpha}$ as a union of two curves, both having end points  $z_1',z_2'$.
Therefore $\tilde{\alpha} \subset \R^2 \setminus \Omega$ and by Theorem \ref{thm:necessity} and \eqref{eq:intspecialcase},
we have
\begin{equation}\label{eq:mainthm_usage}
\begin{split}
    \int_{\tilde{\alpha}} \dist(z,\partial \Omega)^{1-p} \, \d s (z) & \leq C(p,\varepsilon)\|E\|^{2+p+\varepsilon} \int_{\alpha} \dist(z,\partial \Omega)^{1-p} \, \d s (z)\\
    & \leq C(p,\varepsilon)\|E\|^{2+p+\varepsilon}C_{\text{cig-l}}|z_1'-z_2'|^{2-p}.
    \end{split}
\end{equation}

{\color{blue}
Step 2:
}
We now construct the curve $\gamma$ by connecting the sets $\Omega_i$ by suitable line-segments and by using Step 1 for each $\Omega_i$ to connect the entrance and exit points of $\Omega_i$. See Figure \ref{fig:curves} for an illustration of the construction of $\gamma$.

\begin{figure}
    \centering
    \includegraphics[width=0.8\columnwidth]{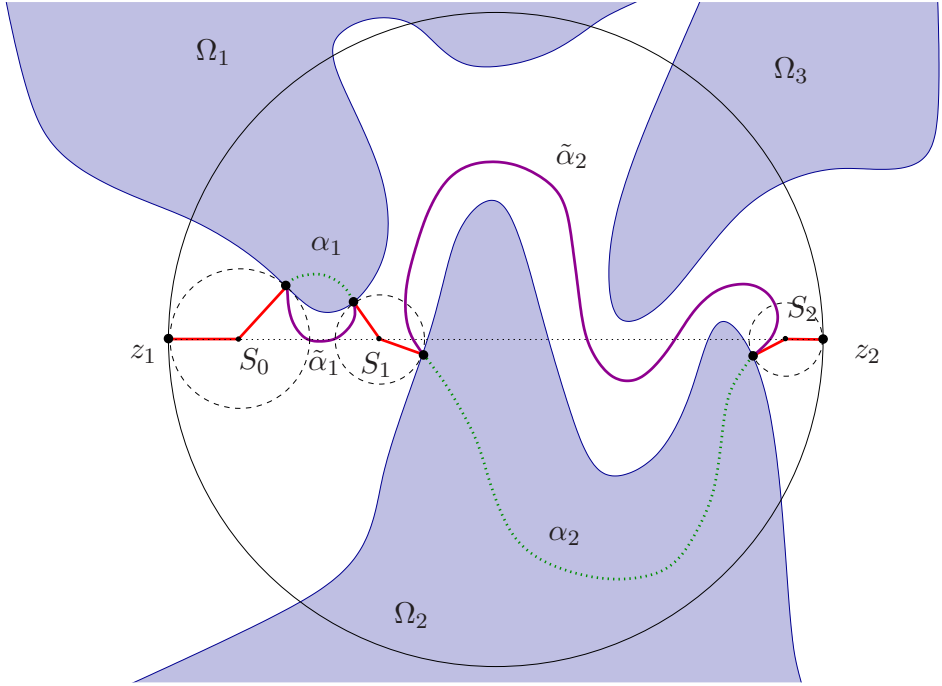}
   \caption{The curve $\gamma$ connecting $z_1$ and $z_2$ satisfying \eqref{eq:curve} is constructed by concatenating radial line segments (giving the curves $S_i$) inside disks completely contained in the complement of $\Omega$
   and curves $\tilde\alpha_i$ that are obtained from Step 1 of the construction as (part of) the boundaries of extensions of sets whose boundary in $\Omega_i$ is $\alpha_i$.
  }
    \label{fig:curves}
\end{figure}

Let us first check that by \eqref{eq:proofcig}, we get an upper bound for the number $k \in \N$  of sets $\Omega_i$.
By the definition of sets $\Omega_i$, $\Omega_i \cap B(x,r) \neq \emptyset$ for all $i = 1, \ldots , k $, so, there exists a curve in $\Omega_i$ satisfying \eqref{eq:proofcig} that starts in $B(x,r)$ and exits $B(x,3r)$ at some point $z \in S^1(x,3r)$ so that
\[
 \dist(z,\partial\Omega) \ge \frac{2r}{C_{\text{cig-l}}}.
\]
Consequently, there exists an arc $S \subset S^1(x,3r) \cap \overline{\Omega_i}$ such that  $\mathcal{H}^1 (S)> \frac{4r}{C_{\text{cig-l}}}$. Hence, $4k \frac{r}{C_{\text{cig-l}}} < 2\pi \cdot 3r$, and so
\begin{equation}\label{eq:omegabound}
k < \frac32\pi C_{\text{cig-l}}.
\end{equation}

Let us then construct a curve connecting $z_1$ and $z_2$  such that  \eqref{eq:curve} holds.
For notational convenience, write $\Omega_0 = \{z_1\}$ and $\Omega_{k+1} = \{z_2\}$.


Define $O_0 = \Omega_0$, $U_0 = \bigcup_{1 \leq i \leq k+1} \Omega_i $ and a continuous function $f_0$ by 
\[
f_0 \colon t \mapsto \dist(O_0, t z_2  + (1-t)z_1) - \dist(U_0, t z_2 + (1-t) z_1), \quad 0\leq t \leq 1.
\]
Let $t_0 = \max\{ t \in [0,1] : f_0(t) = 0 \}$, and 
\[
R_0:=\dist(O_0, t_0 z_2+ (1-t_0)z_1) = \dist(U_0, t_0 z_2 + (1-t_0) z_1).
\]
Denote by $P_0 =  t_0 z_2 + (1-t_0) z_1$. By the selection of $P_0$ and $R_0$ we have $B(P_0,R_0) \subset \R^2 \setminus \Omega$, and there exists $i_0 \in \Big\{ i \in \{1, \ldots , k, k+1 \} : \Omega _i \subset U_0 \Big\}$ 
such that $O_0 \cup \Omega_{i_0} \cup \overline{B}(P_0,R_0)$ is connected. 

We continue by induction. Suppose we have found $P_0,\dots,P_j$, $R_0, \dots, R_j$, $O_0,\dots, O_j$ and $U_0, \dots, U_j$. Replacing $O_j$ with $O_{j+1} = O_j \cup \Omega_{i_0}$ and $U_j$ with $U_{j+1} = U_j \setminus \Omega_{i_0}$ we repeat the above process until $i_0 = k+1$. 

The above process  gives us a (relabeled) sequence $\Omega_0, \ldots ,\Omega_{j+1}$ such that each adjacent pair $\Omega_m$, $\Omega_{m+1}$ may be connected with $S_m=[w_m^2,P_m] \cup [P_m,w_{m+1}^1]$ where 
$w_0^2 = z_1$, $w_{j+1}^1 = z_2$, and 
$w_m^2 \in \partial\Omega_m \cap \partial \Omega \cap \overline{B}(P_m,R_m)$ and $w_{m+1}^1 \in \partial\Omega_{m+1} \cap \partial \Omega\cap \overline{B}(P_m,R_m)$ for the other indices, 
so that
\begin{equation}
     \int_{S_m} \dist(z, \partial\Omega)^{1-p} \, \d s (z) \leq \frac{2}{2-p} \cdot R_m^{2-p} \le \frac{2}{2-p} |z_1-z_2|^{2-p}. \label{eq:connecting_segments}
\end{equation}
For each $m \in \{1, \dots, j\}$ we can connect $w_m^1$ and $w_m^2$ with a curve $\tilde\alpha_m$ given by the special case that satisfies \eqref{eq:mainthm_usage} with the obvious changes of notation.

The final curve $\gamma$ is obtained by the concatenation of the curves 
\[
S_0, \tilde\alpha_1, S_1,  \dots, \tilde\alpha_{j-1}, S_j, \tilde\alpha_j, S_{j+1}.
\]
By the bound \eqref{eq:omegabound} for the number of $\Omega_i$'s, combined with \eqref{eq:mainthm_usage}, \eqref{eq:connecting_segments}, and \eqref{eq:Cigvalue}, we see that the curve $\gamma$ satisfies
\begin{align*}
 \int_\gamma \dist(z, \partial\Omega)^{1-p} \, \d s (z)
 &\leq \frac{3\pi C_{\text{cig-l}}}{2} \cdot \left( \frac{2}{ 2-p} + C(p,\varepsilon)\|E\|^{2+p+\varepsilon}C_{\text{cig-l}}\right) |z_1-z_2|^{2-p} \\
 &\leq C(p,\varepsilon)\|E\|^{2+p+\varepsilon}C_{\text{cig-l}}^2|z_1-z_2|^{2-p}\\
 & \leq C(p,\varepsilon)\|E\|^{\frac{4p}{2-p}+2+p+\varepsilon}|z_1-z_2|^{2-p}\\
 & \leq C(p,\varepsilon)\|E\|^{\frac{4+4p-p^2}{2-p}+\varepsilon}|z_1-z_2|^{2-p}.
\end{align*}
This concludes the proof of the second step and theorem.
\end{proof}

\section{A Sobolev extension domain with large boundary}\label{sec:largeboundary}

In this section we prove Theorem \ref{thm:example} which states the existence of a domain $\Omega\subset \R^3$ such that $\Omega = h(B(0,1))$ for a homeomorphism $h\colon \R^3 \to \mathbb \R^3$, $\dim_{\mathcal H}(\partial\Omega) = 3$ and $\Omega$ is a $W^{1,p}$-extension domain for all $p \in [1,\infty]$. 

We define first the following Cantor set $C\subset [0,1]^3$:   Choose any  strictly increasing sequence of positive numbers $\{\lambda_i\}$ satisfying 
$$\lim_{i\to \infty}\lambda_i=1/2\,\;\;;\;\;\; \prod_{i\geq 1} 2\lambda_i=0. $$
For instance let $\lambda_i:=(1/2)e^{-1/i}$.
Define inductively a family of closed sets $C_n$, with $C_0=[0,1]^3$, so that each $C_n=\bigcup^{8^n}_{i=1}C_{n,i}$ consists of $8^n$ disjoint cubes of sides $l_n=\lambda_1\cdots \lambda_n$, in such a way that for each $i$,  $C_{n+1}\cap C_{n,i}$ is formed by $8$ cubes equally distributed inside $C_{n,i}$ (denote $l_0=1$). Namely, they are at least at a distance  $e_{n+1}=l_n(1-2\lambda_{n+1})/3$ between themselves and also from the boundary of $C_{n,i}$. 

Letting $C=\bigcap_{n\geq 0} C_n$, we get a Cantor set of Hausdorff dimension $3$ but $\mathcal{L}^3(C)=0$. The
fact that $\dim_{\mathcal H}(C) = 3$ follows from $\lambda_i \nearrow \frac12$, while $\mathcal{L}^3(C)=0$ is implied by $\prod_{i\geq 1} 2\lambda_i=0$.
We refer to \cite[Corollary 4]{MVV1992} for further details.


Let us next define tubes that approach our Cantor set.
The tubes will be removed from the open unit cube $(0,1)^3$ to form $\Omega$ so that $C \subset \partial\Omega$. 
Define first $x_{n,i}$ to be the middle point of the upper face of a cube $C_{n,i}$. Given a cube $C_{n,i}$ we denote by $C_{n-1,j(i)}$ the larger cube that contains it from the previous iteration. Define a decreasing sequence of positive constants $c_n$ by setting $c_0=e_1/8\in (0,1)$ and $c_n=c_{n-1}/64$ for all $n\ge 1$. In particular, we then have $c_n\leq e_n/8$ and $c_n\leq l_n$ for every $n\in\mathbb{N}$.

The tubes will be defined as tubular neighbourhoods of  curves $L_{n,i}$ joining $x_{n,i}$ a point $y_{n,i}\in \partial C_{n-1,j(i)}$ on the top face of $C_{n-1,j(i)}$. We require the curves $L_{n,i}$ and points $y_{n,i}$ to satisfy the following conditions:
\begin{itemize}
    \item[(L1)] $L_{n,i}\subset C_{n-1,j(i)}\setminus \interior(C_{n,i})$.
    \item[(L2)] $|y_{n,i}-x_{n-1,j(i)}|\leq c_{n-1}/2$.
    \item[(L3)] $\dist(L_{n,i},L_{n,j})\ge c_n$ for $i \ne j$.
    \item[(L4)] The curves $L_{n,i}$ consist of segments that are parallel to the coordinate axes and have length at least $c_n$.
    \item[(L5)] $L_{n,i}$ approaches $x_{n,i}$ and $y_{n,i}$ perpendicular to the faces of $C_{n,i}$ and $C_{n-1,j(i)}$, respectively.
\end{itemize} 
Using the curves $L_{n,i}$ we then define the tubes as
\[
T_{n,i}=\left\{x\in C_{n-1,j(i)}\setminus \interior(C_{n,i}):\, \dist (x, L_{n,i})\leq  c_{n}/2\right\}.
\]
Next we define
\[
T_n=\bigcup^{8^n}_{i=1}T_{n,i}
\]
for every $n\geq 1$, and finally our domain as
\[
 \Omega=(0,1)^3\setminus \overline{ \bigcup_{n\geq 1} T_n}.
\]
See Figure \ref{fig:example} for a two-dimensional illustration of the construction.

\begin{figure}
    \centering
    \includegraphics[width=0.6\columnwidth]{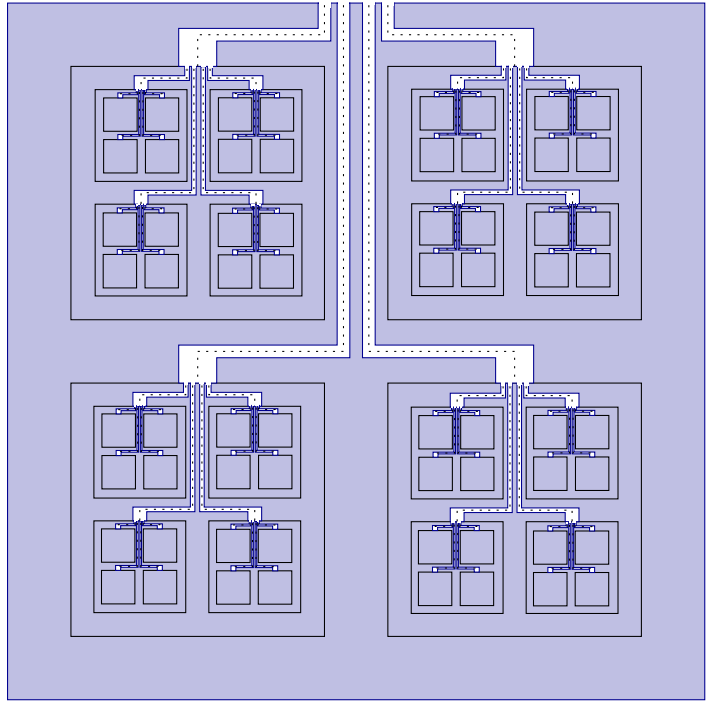}
    \caption{A two-dimensional illustration of the construction of $\Omega$. We start with a unit square and remove disjoint tubes that approach the top sides of Cantor set construction pieces. This two-dimensional version is homeomorphic to the unit ball, but contrary to $n\ge 3$ it is not a Sobolev extension domain.}
    \label{fig:example}
\end{figure}

 By construction, we have
\[
C\subset \partial \bigcup_{n\geq 1} T_n\subset \partial \Omega \subset \partial(0,1)^3 \cup \bigcup_{n\ge 1}\bigcup_{i=1}^{8^n}\partial T_{n,i} \cup C.
\]
Therefore, $\dim_{\mathcal{H}}(\partial\Omega)=3$ and  $\mathcal{L}^3(\partial \Omega)=0$ as required. 
What remains to be checked is that $\Omega$ is homeomorphic to the open unit ball and that it is a Sobolev $W^{1,p}$-extension domain for all $1\leq p\leq \infty$. We prove these separately in the following two lemmata.
 
\begin{lemma}
There exists a homemorphism $h\colon\mathbb{R}^3\to \mathbb{R}^3$ so that $h(B(0,1))=\Omega$.
\end{lemma}

\begin{proof}
  It is enough to prove the existence of a homeomorphism $h\colon\mathbb{R}^3\to \mathbb{R}^3$ so that $h(\Omega)=(0,1)^3$ since the unit cube and the unit ball are homeomorphic under a global homeomorphism.
 
 Define the following decreasing sequence of open sets $\{U_n\}_{n\geq 1}$.
 \[
 U_n=\left\{x\in\mathbb{R}^3:\, \dist \left(x, \overline{\bigcup_{i\geq n} T_i}\right)< c_n\right\}.
 \]
Observe that each $U_n$ has exactly $8^n$ disjoint connected components. Let us label these components as $U_{n,i} \supset T_{n,i}$.
Notice also that
\[
\bigcap_{n\geq 1 } U_n= C.
\]
We now define a sequence of homeomorphisms $h_n=h_{n,1}\circ h_{n,2}\circ\cdots \circ h_{n,8^n}$ so that each $h_{n,i}$ satisfies the following conditions:


\begin{enumerate}
\item[(H1)] For the supports we have
   \[
\textrm{spt}(h_{n,i}):=\{x:\, h_{n,i}(x)\neq x\} \subset U_n.
\]
In particular, for a given $n$ they are pairwise disjoint for different $i$.
\item[(H2)] For every two points $x,y\in U_{n+1}\cap U_{n,i}$ we have    \[
     |h_{n,i}(y)-h_{n,i}(x)|\le |x-y|.
   \]
\item[(H3)] The map $h_{n,i}$ flattens the boundary of the tube $T_{n,i}$ to the top face of $C_{n-1,j(i)}$:
\[
h_{n,i}\left((\partial C_{n-1}\setminus \partial T_{n,i})\cup \overline{(\partial T_{n,i}\setminus\partial C_{n-1})} \right) = \partial C_{n-1}.
\]
\end{enumerate}
 Using the maps $h_n$ we then define
 \[
 h=\lim_{n\to\infty}h_1\circ\cdots \circ h_n.
 \]
 
 Let us next check that $h$ is well defined. On one hand, if $x \notin U_n$ for some $n$, then $h(x) = h_1\circ \cdots \circ h_n(x)$, since by (H1) we have $h_k(x) = x$ when $k > n$. On the other hand, if $x \in \bigcap_{n}U_n$, by the pairwise disjointness of $U_{n,i}$ there exists a unique sequence $(i_n)$ so that $x \in U_{n,i_n}$. Since $\diam(U_{n,i_n}) \to 0$ as $n \to \infty$, by (H2) also
 \[
\diam(h(\overline{U_{n,i_n}})) = \diam(h_1 \circ \cdots \circ h_n(\overline{U_{n,i_n}})) \to 0
 \]
 as $n \to \infty$. Hence $\{h(x)\} = \bigcap_n h(\overline{U_{n,i_n}})$.
 
 Notice that $h$ maps the Cantor set $C$ bijectively to the Cantor set $\bigcap_{n}h(\overline{U_{n}})$. Hence, being a bijection outside the Cantor sets, $h$ is a bijection $\mathbb R^n \to \mathbb R^n$.
  Hence, in order to see that $h$ is a homeomorhism, by domain invariance it is enough to check that $h$ is continuous. This follows by the uniform continuity of the sequence $(h_1\circ\cdots \circ h_n)_n$ of homeomorphisms given by (H1) and (H2). Thus, as the limit of uniformly continuous mappings, $h$ is continuous.
  
  Let us finally observe that $h(\Omega)=(0,1)^3$. This is due to the condition (H3) implying 
  $h(\partial T_n) \subset \partial (0,1)^3$ for all $n$ and hence, by continuity of $h$, we have $h(\partial\Omega) = \partial(0,1)^3$.
  \end{proof}

\begin{lemma}
$\Omega$ is a  Sobolev $W^{1,p}$-extension domain for all $1\leq p\leq\infty$.
\end{lemma}
   
\begin{proof}
Since $\Omega$ is a bounded domain it is enough to prove that it is an $L^{1,p}$-extension domain for the homogeneous norm  since for bounded domains these are the same (see \cite{K1990,HK1991}). We will provide the extension in two steps. For this purpose we divide each of the tubes $T_{n,i}$ into an even number of shorter pieces of tubes to each of which we can extend the Sobolev function from its small neighbourhood.

Recall that our tubes are of the form
\[
T_{n,i}=\left\{x\in C_{n-1,j(i)}\setminus \interior (C_{n,i}):\, \dist (x, L_{n,i})\leq  c_{n}/2\right\}.
\]
We now split each curve  $L_{n,i}$ into finitely many parts $J(i,n)$   
$$L_{n,i}=\bigcup^{J(i,n)}_{j=1}L^{j}_{n,i} $$ so that the following four properties hold:
\begin{itemize}
    \item[(P1)] $J(n,i)$ is an even number.
    \item[(P2)] $\ell (L^{j}_{n,i})\in [2c_n,6c_n]$ for every $j=1,\dots, J(n,i)$.
    \item[(P3)] $L^{j}_{n,i}\cap L^{j+1}_{n,i}$ is just one point for every $j=1,\dots, J(n,i)-1$.
    \item[(P4)] $L^{1}_{n,i}$ touches $\partial C_{n-1,j(i)}$ and $  L^{J(n,i)}_{n,i}$ touches $C_{n,i}$.
\end{itemize}
The property (P1) together with (P2) can be satisfied because by the definition of $c_n$, we have $\ell(L_{n,i}) \ge 8c_n$. The condition (P3) and (P4) just say that the curves follow one after another in the desired direction. Using the shorter curves $L_{n,i}^j$ we then write $T_{n,i}=\bigcup^{J(n,i)}_{j=1} T^{j}_{n,i}$, where each shorter tube $T^{j}_{n,i}$ is the closure of the set of points of $T_{n,i}$ which are closer to $L^{j}_{n,i}$, for every $j$.

We define the following open sets from which we extend a given Sobolev function to the corresponding tube. If $j$ is odd, we set
\[
U^{j}_{n,i}=\left\{x\in \interior(C_{n-1,j(i)}):\, \dist (x,L^{j}_{n,i})<2c_n,\text{and}\;x\, \text{is closer to} \, L^{j}_{n,i} \,\text{than to other }\, L^{j'}_{n,i} \right\},
\]
and if $j$ is even, we set
\[
U^{j}_{n,i}=\left \{x\,:\, \dist (x,T^{j}_{n,i})< c_{n+1}\right \}.
\]

By the assumptions (L4) and (P2), there exists a constant $L$ so that for every $n,i$ and $j$ there exists a map $f_{n,i}^j\colon \mathbb R^3 \to \mathbb R^3$ which is a composition of an  $L$-biLipschitz map and a similitude so that 
\[
f_{n,i}^j(U_{n,i}^j) = U_{\text{odd}} \quad \text{and}\quad
f_{n,i}^j(T_{n,i}^j) = T,
\]
for $j$ odd, and 
\[
f_{n,i}^j(U_{n,i}^j) = U_{\text{even}} \quad \text{and}\quad
f_{n,i}^j(T_{n,i}^j) = T,
\]
for $j$ even, where
\[
T = \left\{x=(x_1,x_2,x_3)\,:\, x_1\in [0,1]\,,\, \sqrt{x_2^2+x_3^2} \le 1\right\},
\]
\[
U_{\text{odd}} = \left\{x=(x_1,x_2,x_3)\,:\, x_1\in (0,1)\,,\, \sqrt{x_2^2+x_3^2} < 2\right\},
\]
and
\[
U_{\text{even}} = \left\{x\,:\, \dist(x,T)< 1\right\}.
\]
Now, for instance by Jones' theorem \cite{jo1981}, there exists an extension operator 
\[
E_{\text{odd}} \colon L^{1,p}(U_{\text{odd}}\setminus T) \to L^{1,p}(U_{\text{odd}}),
\]
since $U_{\text{odd}}$ is an $(\varepsilon,\delta)$-domain. Consequently, the norms of the operators for odd $j$
\[
E_{n,i}^j \colon L^{1,p}(U_{n,i}^j \setminus T_{n,i}^j) \to L^{1,p}(U_{n,i}^j) \colon u \mapsto (E_{\text{odd}}(u\circ (f_{n,i}^j)^{-1})) \circ f_{n,i}^j
\]
are uniformly bounded. (Notice that the $L$-biLipschitz part of $f_{n,i}^j$ changes the norms only by a constant, whereas the similitude parts cancel out their effect on the norm since we use the homogenous norm.)

Similarly, for $j$ even, there exists an extension operator
\[
E_{\text{even}} \colon L^{1,p}(U_{\text{even}}\setminus T) \to L^{1,p}(U_{\text{even}}),
\]
and so each of the operators for even $j$
\[
E_{n,i}^j \colon L^{1,p}(U_{n,i}^j  \setminus T_{n,i}^j) \to L^{1,p}(U_{n,i}^j) \colon u \mapsto (E_{\text{even}}(u\circ (f_{n,i}^j)^{-1})) \circ f_{n,i}^j
\]
are also uniformly bounded.

\begin{figure}
    \centering
    \includegraphics[width=0.95\columnwidth]{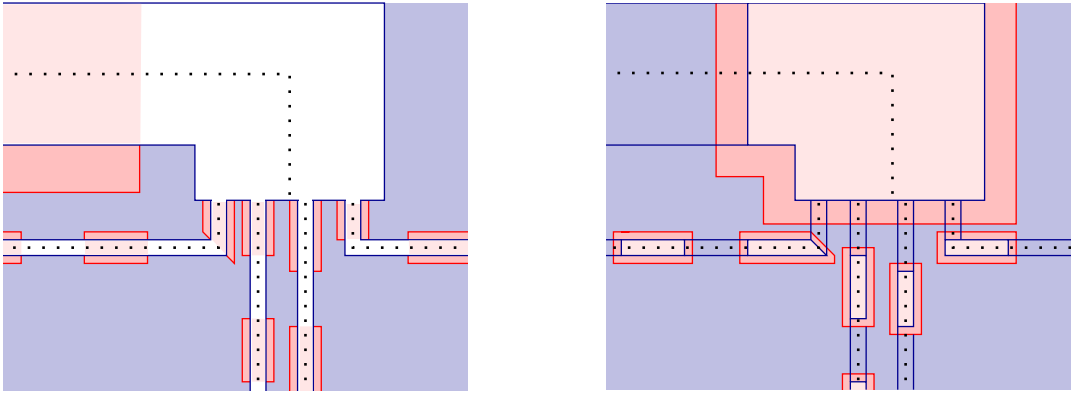}
    \caption{A two-dimensional illustration of the extension operator $E^1$ (on the left) that fills in every second piece of the tubes, and the extension operator $E^2$ (on the right) that fills in the rest of the pieces.}
    \label{fig:extensions}
\end{figure}

Next we see from the assumption (L3) that the collection $\{U_{n,i}^j\}_{j \text{ odd}}$ is pairwise disjoint. Hence, the extension operator 
     \[
     E^{1}\colon  L^{1,p}(\Omega)\to L^{1,p}\Bigg( (0,1)^3\setminus \overline{\bigcup_{\substack{i,n\\ j\;\text{even}}} T^{j}_{n,i}}\Bigg)
     \]
     defined by 
    \[
    E^1u(x)=\begin{cases}
    E^{j}_{n,i}(u|_{U^{j}_{n,i}\setminus T^{j}_{n,i}}) (x), & \text{if } x\in  U^{j}_{n,i}\;\text{with}\; j\;\text{odd},\\
    u(x), & \text{otherwise},
    \end{cases}
    \]
    is bounded. Then we use again the assumption (L3) to notice that also the collection $\{U_{n,i}^j\}_{j \text{ even}}$ is pairwise disjoint. Therefore, also the extension operator
\[
E^{2}\colon L^{1,p}\Bigg( (0,1)^3\setminus \overline{\bigcup_{\substack{i,n\\ j\;\text{even}}} T^{j}_{n,i}}\Bigg)\to L^{1,p}((0,1)^3\setminus C)
\]

defined by 
\[
E^2u(x)=\begin{cases}
    E^{j}_{n,i}(u|_{U^{j}_{n,i}\setminus T^{j}_{n,i}}) (x),& \text{if } x\in  U^{j}_{n,i}\;\text{with}\; j\; \text{even},\\
    u(x), & \text{otherwise},
    \end{cases}
    \]
    is bounded. See Figure \ref{fig:extensions} for an illustration of the extension operators $E^1$ and $E^2$.

Finally, we observe that $(0,1)^3$ is an extension domain (with some extension operator $E^3$) and the set $C$ is removable for Sobolev functions since its projection to any coordinate plane has zero two-dimensional measure. Thus, $E^3 \circ E^2 \circ E^1 \colon L^{1,p}(\Omega) \to L^{1,p}(\mathbb R^3)$ is a bounded extension operator.
\end{proof}

\end{document}